\documentclass[letterpaper,11pt]{article}
\RequirePackage{amsmath, amsthm, amssymb, mathtools}
\RequirePackage{comment}
\RequirePackage{graphicx}
\RequirePackage{fancyvrb}
\RequirePackage[usenames,dvipsnames,svgnames,table]{xcolor}
    \definecolor{linkcolor}{RGB}{0,0,128}
\RequirePackage[pdfpagelabels,plainpages=false,hypertexnames=true,colorlinks=true,linkcolor=linkcolor,anchorcolor=linkcolor,citecolor=linkcolor,filecolor=linkcolor,menucolor=linkcolor,runcolor=linkcolor,urlcolor=linkcolor,pdfborder={0 0 0}]{hyperref}
\RequirePackage[scale=0.95]{sourcecodepro}
\RequirePackage{mathpazo}

\RequirePackage{dsfont}
\RequirePackage{parskip}
\RequirePackage{tikz}
\usetikzlibrary{positioning}

\RequirePackage[backend=biber,backref=true,style=numeric,giveninits=true,natbib=true,maxbibnames=12]{biblatex}
\usepackage{orcidlink}
\usepackage{enumerate,enumitem,romanbar}

\RequirePackage{microtype}
\usepackage{newfloat}
\DeclareFloatingEnvironment[
    fileext=los,
    listname={List of Algorithms},
    name=Algorithm,
    placement=tbhp,
    within={},
]{algorithm}

\theoremstyle{plain}
\newtheorem{prototheorem}{Theorem}

\newtheorem{theorem}[prototheorem]{Theorem}
\newtheorem{lemma}[prototheorem]{Lemma}

\newtheorem{example}[prototheorem]{Example}

\theoremstyle{remark}

\newtheorem{remark}{Remark}
\newtheorem{assumption}{Assumption}

\DeclarePairedDelimiter\floor{\lfloor}{\rfloor}

\newcommand{\norm}[1]{\left\lVert#1\right\rVert}
\newcommand{\R}{\mathbb{R}}

\newcommand{\Z}{\ensuremath \mathbb{Z}}

\newcommand{\opnorm}[1]{\left\| #1 \right\|_{\mathsf{op}}} 
\newcommand{\wnorm}[1]{\left\| #1 \right\|_{\mathsf{w}}} 
\newcommand{\wonenorm}[1]{\left\| #1 \right\|_{\mathsf{w},1}}

\newcommand{\rn}[1]{\Romanbar{#1}}

\title{Randomized Runge-Kutta-Nystr\"{o}m Methods for Unadjusted Hamiltonian and Kinetic Langevin Monte Carlo}
\author{Nawaf Bou-Rabee\thanks{Department of Mathematical Sciences, Rutgers University, \href{mailto:nawaf.bourabee@rutgers.edu}  {\texttt{nawaf.bourabee@rutgers.edu}} \orcidlink{0000-0001-9280-9808}}
\and
Tore Selland Kleppe\thanks{Department of Mathematics and Physics, University of Stavanger, 
\href{mailto:mtore.kleppe@uis.no}{\texttt{tore.kleppe@uis.no}} \orcidlink{0000-0001-8469-908X}}
}

\begin{document}

\maketitle
\begin{abstract}
\noindent
We introduce $5/2$- and $7/2$-order $L^2$-accurate randomized Runge-Kutta-Nystr\"{o}m methods, tailored for approximating Hamiltonian flows within non-reversible Markov chain Monte Carlo samplers, such as unadjusted Hamiltonian Monte Carlo and unadjusted kinetic Langevin Monte Carlo. We establish quantitative $5/2$-order $L^2$-accuracy upper bounds under gradient and Hessian Lipschitz assumptions on the potential energy function. The numerical experiments demonstrate the superior efficiency of the proposed unadjusted samplers on a variety of well-behaved, high-dimensional target distributions.
\end{abstract}

\section{Introduction}

Measure-preserving ODEs, SDEs, and PDMPs play a fundamental role in the construction and analysis of non-reversible Markov chain Monte Carlo (MCMC) kernels \cite{RoTw1996B,DiHoNe2000,MaStHi2002,ELi2008,EbMa2019,Deligiannidis2021,kleppe2022,BoEb2022}.  Among time integrators for the corresponding flows, randomized time integrators have stood out in terms of complexity  \cite{shen2019randomized,ErgodicityRMMHYB,Cao_2021_IBC,BouRabeeMarsden2022,BouRabeeSchuh2023B,BouRabeeOberdoerster2024}.   The key idea behind these randomized time integrators is to replace the deterministic quadrature rules typically used in each time integration step with a Monte Carlo integration rule, akin to stochastic gradient descent in machine learning \cite{bottou2012stochastic}, which typically uses only a single sample per integration step. This approach not only reduces the regularity requirements on the coefficients of the underlying ODE, SDE or PDMP, but also significantly improves the asymptotic bias properties of the resulting Markov kernel, leading to reduced computational complexity \cite{BouRabeeMarsden2022,BouRabeeSchuh2023B}.  

\smallskip 

Randomized time integrators were first developed in the early 1990s by Stengle for Carath\'{e}odory differential equations (i.e., ODEs with vector fields that are  Lipschitz continuous in space but only measurable in time). Stengle proposed a family of randomized versions of the deterministic Runge-Kutta method \cite{stengle1990numerical}, along with a  convergence analysis of a representative method from this family  \cite{stengle1995error}.  This analysis was later refined and generalized in subsequent works  \cite{jentzen2009random,kruse2017error,bochacik2021randomized}. The general finding of these works is that in low regularity settings, randomized numerical methods offer better complexity than their deterministic counterparts.  Randomized versions of the Euler-Maruyama and Milstein schemes have also been proposed and studied for SDEs with low-regularity drift coefficients \cite{kruse2017randomized,bencheikh2022convergence}.

\smallskip 

To our knowledge, there has been limited work on developing randomized versions of Runge-Kutta-Nystr\"{o}m (RKN) schemes. RKN schemes are a specialized class of numerical methods designed to solve second-order differential equations, particularly those of the form  $\ddot{q}_t = \mathbf{F}(q_t, t)$ \cite{hairer1975theory,HaLuWa2010}.  The term \emph{Runge-Kutta-Nystr\"{o}m} originates from the contributions of Carl Runge and Wilhelm Kutta, who developed Runge-Kutta methods for general ODEs, and Ernst Nystr\"{o}m, who extended these methods to second-order systems. Unlike standard Runge-Kutta methods, which are tailored for first-order systems, RKN schemes take advantage of the second-order structure, where the position evolves with constant velocity to leading order.  This makes it easier to discretize the position variable, and by treating the position and velocity variables separately, RKN schemes can efficiently approximate second-order systems.  This leads to greater accuracy and efficiency in applications such as molecular dynamics, celestial mechanics, and Hamiltonian Monte Carlo.   A widely used example of an RKN scheme is the Verlet (or leapfrog) integrator \cite{HaLuWa2010,LeRe2004,BoSaActaN2018}.

\smallskip

The current state-of-the-art for randomized time integration of second-order measure-preserving processes are 3/2-order $L^2$-accurate integrators, under the assumption of Lipschitz regularity for the underlying drift fields \cite{ErgodicityRMMHYB,Cao_2021_IBC,BouRabeeMarsden2022,BouRabeeSchuh2023B}. This naturally raises the question of whether higher-order randomized time integrators can be developed for such processes, and what their asymptotic bias properties might be.  The goal of this paper is to demonstrate that higher-order randomized RKN methods can be constructed and that they can offer improved properties for the corresponding unadjusted  kernels.  Specifically, we introduce 5/2- and 7/2-order randomized Runge-Kutta-Nystr\"{o}m (rRKN) schemes tailored for Hamiltonian flows.  We provide a thorough investigation of both the numerical performance and theoretical properties of these higher-order methods, focusing on their application to unadjusted Hamiltonian and kinetic Langevin Monte Carlo.

\smallskip 
Another key contribution of this work is the numerical demonstration that the 5/2- and 7/2-order rRKN methods, when applied to unadjusted Hamiltonian and  kinetic Langevin Monte Carlo, can significantly reduce computational cost while maintaining the same level of tolerated bias. These higher-order rRKN methods outperform both the widely used Verlet integrator \cite{BoSaActaN2018} and  the more recently proposed stratified Monte Carlo integrator of  \cite{BouRabeeMarsden2022}, particularly when the target distribution is well-behaved. In scenarios where the underlying potential is gradient and Hessian Lipschitz continuous, and the Hessian matrix is well-conditioned, the high-order accuracy of rRKN schemes translates directly into fewer transition steps, which leads to lower computational cost.  This efficiency makes the rRKN methods a compelling choice for sampling in such settings, offering a clear advantage over existing methods.

\smallskip 

The paper makes two main theoretical contributions: (i) a rigorous proof of quantitative $L^2$-accuracy bounds for the rRKN 2.5 scheme under only gradient and Hessian Lipschitz assumptions (Theorem~\ref{thm:rRKN_L2_accuracy}); and (ii) analogous results for an unadjusted kinetic Langevin chain (Theorem~\ref{thm:ukLa_L2_accuracy}).  A central aspect of the rRKN 2.5 scheme in \eqref{eq:rRKN2pt5} is the use of a triangular random variable in its update rule, which  is essential in achieving its improved $L^2$-accuracy compared to the Verlet integrator, provided the specified regularity conditions are met. 

\smallskip

These proofs build on the well-established \emph{Fundamental Theorem for $L^2$-Convergence of Strong Numerical Methods for SDEs}, which states that if $p_1 \ge p_2 + 1/2 $ and $p_2 \ge 1/2$ are the local order of mean and mean-square accuracy (respectively), then the global $L^2$-accuracy of the method is of order $p_2 - 1/2$ \cite[Theorem 1.1.1]{Milstein2021}.  For an rRKN integrator, the analogous heuristic is that if  the local bias  is of order $p$ and the local mean-squared error is of order $q$, with $p \ge q+1/2$, then the global $L^2$-error  scales as $q-1/2$.


\smallskip

Specifically, the proofs rely  on four main components: (a) Lipschitz regularity of the exact Hamiltonian flow (Lemma~\ref{lem:lipschitz_exact_flow}); (b) numerical stability of  the rRKN integrator (Lemmas~\ref{lem:apriori} or~\ref{lem:apriori_ukLa}); 
(c) one-step mean accuracy of the rRKN integrator relative to the exact Hamiltonian flow (Lemma~\ref{lem:rRKN_local_mean_error}); and (d) one-step mean-squared accuracy of the rRKN integrator relative to the exact Hamiltonian flow (Lemma~\ref{lem:rRKN_local_L2_error}). By adapting the proofs of (b)-(d), and under appropriate regularity conditions, the approach presented in this paper can,  in principle, be extended to prove quantitative $L^2$-accuracy bounds for arbitrary-order rRKN integrators.  

\smallskip 

The paper is organized as follows.  The rRKN 2.5/3.5  schemes are introduced in Section \ref{sec:algorithms}, and as applications, unadjusted Hamltonian and kinetic Langevin Monte Carlo chains involving these randomized time integrators are described in Section \ref{sec:uHMC} and Section \ref{sec:ukLa}.  The main numerical findings, demonstrating the superior efficiency of the new schemes compared to the Verlet integrator for representative target distributions, are detailed in Section \ref{sec:numerics}.  The essential observation is that the new schemes are consistently more efficient than Verlet.  The main theoretical results are given in Section~\ref{sec:theory}. 

\section*{Acknowledgements}

We would like to express our sincere gratitude to the referees for their careful review, thoughtful feedback, and insightful comments.   N.B.~thanks Milo Marsden for useful discussions.


N.~Bou-Rabee has been partially supported by NSF Grant No.~DMS-2111224.

\section{The Randomized Algorithms} \label{sec:algorithms}

  This section presents randomized versions of the Runge-Kutta-Nystr\"{o}m method for approximating Hamiltonian flows and considers applications of these schemes to (1) unadjusted Hamiltonian Monte Carlo \cite{BouRabeeSchuh2023,BouRabeeEberle2023,monmarche2022hmc} and (2) unadjusted kinetic Langevin chains \cite{cheng2018underdamped,dalalyan2020sampling, monmarche2021high,monmarche2022hmc}. 
For a unified and comprehensive treatment of unadjusted MCMC, see \cite{DurmusEberle2021}.    For background on Runge-Kutta-Nystr\"{o}m and related schemes for Hamiltonian problems see \cite{sanz1992symplectic,SaCa1994,LeRe2004,HaLuWa2010}.

\subsection{Randomized Runge-Kutta-Nystr\"{o}m 2.5 \& 3.5}

Let $U: \mathbb{R}^d \to \mathbb{R}$ be a continuously differentiable potential energy and $\mathbf{F} \equiv -\nabla U$ be the corresponding force.  
  For any $t>0$, denote by  $\varphi_t(x,v) \, := \, (q_t(x,v),v_t(x,v))$ the \emph{exact} Hamiltonian flow which satisfies
\begin{equation} \label{exact}
\frac{d}{dt} {q}_t  \,  =\,     {v}_t \;, \quad  \frac{d}{dt} {v}_t \,  =\, \mathbf{F}({q}_t) \;, \quad (q_0, v_0) \ = \  (x,v) \ \in \ \mathbb{R}^{2d} \;. 
\end{equation}
To approximate the flow of \eqref{exact}, we  use  \emph{randomized Runge-Kutta-Nystr\"{o}m} (rRKN) schemes, which are a generalization of the 3/2-order scheme in \cite{BouRabeeMarsden2022} to higher-order accuracy under more restrictive regularity assumptions on $U$.

\medskip

Let $h>0$ be a time step size and define the evenly spaced time grid $\{ t_k := k h \}_{k \in \mathbb{N}_0}$.     Introduce the sequence of random variables \[ ( \mathcal{U}_{i} )_{i \in \mathbb{N}_0} \overset{i.i.d.}{\sim} \text{Triangular}(0,1) \] where $\text{Triangular}(0,1)$ denotes the triangular distribution with mode at $1$, i.e., the continuous probability distribution whose support is confined to $[0,1]$ and with probabiliy density function given by $f(\mathsf{x}) = 2 \mathsf{x}$ for all $\mathsf{x} \in [0,1]$.   For any $i \in \mathbb{N}_0$, a step of the \emph{5/2-order rRKN} time integrator from $t_i$ to $t_{i+1}$ is given by  \begin{equation} \label{eq:rRKN2pt5}
\begin{aligned}
	\Tilde{Q}_{t_{i+1}}  \, &= \, \Tilde{Q}_{t_i} + h \Tilde{V}_{t_i} + \frac{h^2}{2}  \Tilde{\mathbf{F}}_{t_i}^- + \frac{h^2}{6 \mathcal{U}_i} \left(  \Tilde{\mathbf{F}}_{t_i}^+ - \Tilde{\mathbf{F}}_{t_i}^- \right)  \\ \Tilde{V}_{t_{i+1}} \ &= \  \Tilde{V}_{t_i} + h \tilde{\mathbf{F} }_{t_i}^-  + \frac{h}{2 \mathcal{U}_i} \left(  \Tilde{\mathbf{F}}_{t_i}^+ - \Tilde{\mathbf{F}}_{t_i}^- \right)   \;,
 \end{aligned}
\end{equation}
where $(\tilde{Q}_0, \tilde{V}_0) = (x,v) \in \mathbb{R}^{2d}$, \[ \Tilde{\mathbf{F}}_{t_i}^-  = \mathbf{F} (\Tilde{Q}_{t_i})  \quad \text{and} \quad \Tilde{\mathbf{F}}_{t_i}^+  = \mathbf{F} (\Tilde{Q}_{t_i} + (\mathcal{U}_i h)  \tilde{V}_{t_i} + (1/2) (\mathcal{U}_i h)^2 \Tilde{\mathbf{F}}_{t_i}^- ) \;.
\]
Equivalently, this discrete-time evolution can be described in terms of a flow map as follows
\begin{align}
 \label{rRKN_flow}
& (\tilde{Q}_{t_{i+1}}, \tilde{V}_{t_{i+1}}) \ = \ \Theta_h(\mathcal{U}_{i})(\tilde{Q}_{t_i}, \tilde{V}_{t_i})
\;,  \quad  (\tilde{Q}_0, \tilde{V}_0) \ = \ (x,v)  \ \in \  \mathbb{R}^{2d} \;, \quad \text{where} \\
& \Theta_h(\mathsf{u})(x,v) \ := \ \Big(x+ h v + \frac{h^2}{2} \mathbf{F}^-+ \frac{h^2}{6 \mathsf{u}} (\mathbf{F}^+ - \mathbf{F}^-) , 
v + h \mathbf{F}^- + \frac{h}{2 \mathsf{u}} ( \mathbf{F}^+- \mathbf{F}^- )  \Big)  \;, \\ & \text{and} \quad \mathsf{u} \ \in \ [0,1] \;. \nonumber
 \end{align}
Here, we have defined
$\mathbf{F}^- \, := \, \mathbf{F}(x)$ and $\mathbf{F}^+ \, := \, \mathbf{F}(x + (\mathsf{u} h) v + (1/2) (\mathsf{u} h)^2 \mathbf{F}^-)$.    This scheme is almost surely locally fourth order in position and third order in velocity, but the local bias is fourth order, which  heuristically implies $5/2$-order $L^2$-accuracy.  The general heuristic is that if  the local bias of an rRKN integrator is of order $p$ and the local mean-squared error is of order $q$, with $p \ge q+1/2$, then the global $L^2$-error will scale as $q-1/2$. Theorem~\ref{thm:rRKN_L2_accuracy} formalizes this heuristic into a rigorous result.

Remarkably, with only one additional force evaluation per integration step, we obtain a \emph{7/2-order rRKN} time integrator  
\begin{equation}
\label{eq:rRKN3pt5}
\begin{aligned}
    \tilde{Q}_{t_{i+1}} &= \tilde{Q}_{t_i} + h \tilde V_{t_i} + h^2 \left( \frac{1}{6} \tilde{\mathbf F}_{t_i}^- + \frac{1}{3} \tilde{\mathbf{F}}_{t_i}^*  \right) \;, \\
     \tilde V_{t_{i+1}} &= \tilde V_{t_{i}} +
     h \left( \frac{1}{6} \tilde{\mathbf F}_{t_i}^- + \frac{2}{3} \tilde{\mathbf{F}}_{t_i}^* + \frac{1}{6}  \tilde{\mathbf{F}}_{t_{i+1}}^- \right)  \;,
     \end{aligned}
\end{equation}
where $\tilde{\mathbf{F}}_{t_i}^* = 
    \mathbf{F} \left( \tilde{Q}_{t_i} + \frac{h}{2}\tilde V_{t_i} + h^2\left[ \frac{3}{32} \tilde{\mathbf F}_{t_i}^- + \frac{1}{32} \tilde{\mathbf{F}}_{t_i}^+  \right] \right)$.  This scheme is almost surely locally fifth order in position and fourth order in velocity, but the local bias is fifth order, which heuristically implies $7/2$-order $L^2$-accuracy.   These statements, of course, only  hold rigorously under  sufficient regularity conditions.   
    
\medskip

From the viewpoint of MCMC, the exact Hamiltonian flow in \eqref{exact} has several nice properties (including volume and energy preservation), and as a consequence, it preserves  \begin{equation}
\label{eq:muk} \mu(dz)   \ \propto \ e^{-H(z)} \; dz \;, ~~\text{with} ~~  H(z)\ =\ \frac{1}{2} |v|^2\ +\ U(x) \;, ~~ z \ = \ (x,v) \ \in \  \mathbb{R}^{2d} \;.
\end{equation}
However, the flow also has a serious defect, namely that it admits infinitely many invariant measures. This defect motivates  velocity randomization in 
Hamiltonian Monte Carlo (HMC) \cite{Ne2011, BoSaActaN2018} and also considering Markov chains based on discretizations of the kinetic Langevin diffusion  \cite{cheng2018underdamped,cheng2018sharp,dalalyan2020sampling,shen2019randomized,monmarche2021high}.

\subsection{Application to an Unadjusted Hamiltonian Monte Carlo Chain}

\label{sec:uHMC}


Let $\mathcal{N}(0,I_d)$ denote a standard multivariate normal distribution  with zero mean vector and $d \times d$ identity covariance matrix $I_d$.  
As a first application of rRKN, we consider an unadjusted HMC chain that uses rRKN for the underlying Hamiltonian flow.  For simplicity, we only consider full velocity randomization (i.e., the initial velocity per transition step is sampled anew from $\mathcal{N}(0,I_d)$), fixed duration $T>0$, and time step size $h > 0$ such that $T/h \in \mathbb{N}$.  For $x \in \mathbb{R}^d$, a transition step of \emph{unadjusted HMC} (uHMC) with the rRKN flow is then given by
\begin{equation}
X^u(x) \ = \  \tilde{Q}_T(x, \xi) \;, \quad \text{where} \quad \xi \sim \mathcal{N}(0,I_d) \;.
\label{uHMC}
\end{equation}
We stress that: (i) the random initial velocity $\xi$ is independent of state of the chain and of the random variables $( \mathcal{U}_{i} )_{i \in \mathbb{N}_0} \overset{i.i.d.}{\sim} \text{Triangular}(0,1)$ used in the rRKN flow; and, (ii) unadjusted HMC has an asymptotic bias.  
 
\subsection{Application to an Unadjusted Kinetic Langevin Chain}

\label{sec:ukLa}

Another application of rRKN  is to unadjusted chains based on the kinetic Langevin diffusion.  We consider kinetic Langevin diffusions $(X_t,V_t)_{t \ge 0}$  satisfying   \begin{equation}
\label{eq:kLd}
d X_t \, = \, V_t \, dt \;, ~~ d V_t \, = \,  \mathbf{F}(X_t) \,  dt - \gamma V_t \, dt + \sqrt{2 \gamma} \,  d B_t \;, ~~ (X_0, V_0) \, = \, (x,v)  \, \in \, \mathbb{R}^{2d} \;,
\end{equation}   
where $(B_t)_{t \ge 0}$ is a standard $d$-dimensional Brownian motion and $\gamma>0$ is a friction factor.  A key property of the kinetic Langevin diffusion is that under mild conditions on $U$ it has a \emph{unique} invariant probability measure given by $\mu$. Many discretizations of \eqref{eq:kLd} are possible \cite{ScLeStCaCa2006,LeRoSt2010, BoVa2010,Bo2014}. We focus on a  Lie-Trotter splitting   \cite{BuDoPa2007,Bo2014,BePiSaSt2011},  where the splitting components are given by Hamilton's equations for $H$ and Ornstein-Uhlenbeck (OU) equations in velocity.  As previously mentioned, the  former flow is $\varphi_t(x,v) \, = \, (q_t(x,v),v_t(x,v))$; and the latter flow  is determined by the OU-map:  \begin{align}
 \label{eq:OU}
    O_{h}(\mathsf{b}) (x,v) \ &= \ (x, e^{-\gamma h } v + (1- e^{-2 \gamma h } )^{\frac{1}{2}}  \, \mathsf{b}  ) \;,   \quad \mathsf{b} \ \in \ \mathbb{R}^d \;.
    \end{align}
Let $\{ \xi_i \}_{i \in \mathbb{N}_0} \overset{i.i.d.}{\sim} \mathcal{N}(0, I_d)$ and independent of $( \mathcal{U}_{i} )_{i \in \mathbb{N}_0} \overset{i.i.d.}{\sim} \text{Triangular}(0,1)$.  The OU flow map in \eqref{eq:OU} and the rRKN flow map in \eqref{rRKN_flow} are then combined as follows \begin{equation}
Z_{t_{i+1}}^u \ = \ O_{h}(\xi_i) \circ \Theta_{h}(\mathcal{U}_i)  (Z_{t_i}^u)  \;, \quad Z_0^u \ = \ (x,v) \ \in \ \mathbb{R}^{2 d} \;.
\label{OrRKN} 
\end{equation}
 The transition step of the corresponding \emph{unadjusted kinetic Langevin algorithm} (ukLa) is defined by \begin{equation}
Z^u(z) \ = \ O_{h}(\xi) \circ \Theta_{h}(\mathcal{U})  (z)  \;, ~~ \text{where} ~~  \xi \sim \mathcal{N}(0,I_d) \;,  ~~ \mathcal{U} \sim \operatorname{Triangular}(0,1)  \;, 
\label{ukLa}
\end{equation} where $z=(x,v) \in \mathbb{R}^{2d}$
and we stress that $\xi$ and $\mathcal{U}$ are mutually independent random variables.
In order to quantify the asymptotic bias of ukLa w.r.t.~$\mu$, it suffices to compare the ukLa chain with the chain generated by the \emph{exact splitting}
\begin{equation}
Z_{t_{i+1}}^e \ = \ O_{h}(\xi_i) \circ \varphi_{h} (Z_{t_i}^e)\;, \quad Z_0^e  \ = \ (x,v) \ \in \ \mathbb{R}^{2 d} \;, 
\label{OX} 
\end{equation}
where the exact Hamiltonian flow is used per integration step. Since each  sub-step in \eqref{OX} leaves the probability measure $\mu$ invariant,   the corresponding chain leaves $\mu$ invariant.

\section{Numerical Illustrations}

\label{sec:numerics}

To illustrate the proposed schemes, and their effect on resulting unadjusted MCMC algorithms, some limited numerical studies have been performed. Specifically, the contending schemes were Verlet, the stratified integrator of \cite{BouRabeeMarsden2022} (referred to as sMC), and the two rRKN schemes proposed here.

The former set of experiments consider analytically tractable example target distributions. Specifically, the standard multivariate Gaussian distributions with dimension 10, 100 and 1000 (referred to as G10, G100, G1000), along with three non-Gaussian bivariate distributions NG1, NG2, NG3, characterized by
NG1: $q_1 \sim N(0,1),\;q_2|q_1 \sim N(q_1^2/4,1)$,
NG2: $U(q) = (1+(q_{1}^{2}+q_{2}^{2})/8)^{5}$, and 
NG3: $U(q) = (1-q_1^2)^2 + (q_2-q_1)^2/2$.
Here, NG2 is the zero mean, unit scale matrix bivariate $t_8$ distribution, and NG3 is a bi-modal distribution.

Let $\pi(q)$ denote the density of the target distribution. As a measure of the performance of uHMC (\ref{uHMC}) and ukLa (\ref{ukLa}), the bias of 
\begin{equation}
    \frac{1}{n} \sum_{k=1}^n \log \pi(\mathcal{X}_k),
    \label{numeric_estimator}
\end{equation}
relative to $\int [\log \pi(q)] \pi(q) dq$ is used. For uHMC, $\mathcal{X}_k=X^u(\mathcal{X}_{k-1}), \; k \in \{ 1,\dots,n \}$, while for ukLa, $(\mathcal{X}_k,\mathcal{V}_k) = (Z^u)^{T/h}(\mathcal{X}_{k-1},\mathcal{V}_{k-1}),\; k \in \{ 1,\dots,n \}$, where in both cases $\mathcal{X}_0\sim \pi$ and in addition $\mathcal{V}_{0}\sim \mathcal N(0,I_d)$. 
Here, $X^u$ and $Z^u$ are defined by \eqref{uHMC} and \eqref{ukLa}, respectively.  
Throughout, $n=1000$ and $R$ independent estimators (\ref{numeric_estimator}) were used for estimating the bias associated with each combination of target distribution, scheme and selected numbers of time steps $T/h$. In all cases, $T=\pi/2$ is used for uHMC and $T=\pi,\;\gamma=2$ for ukLa.

\begin{figure} 
\centering
\includegraphics[width=0.8\textwidth]{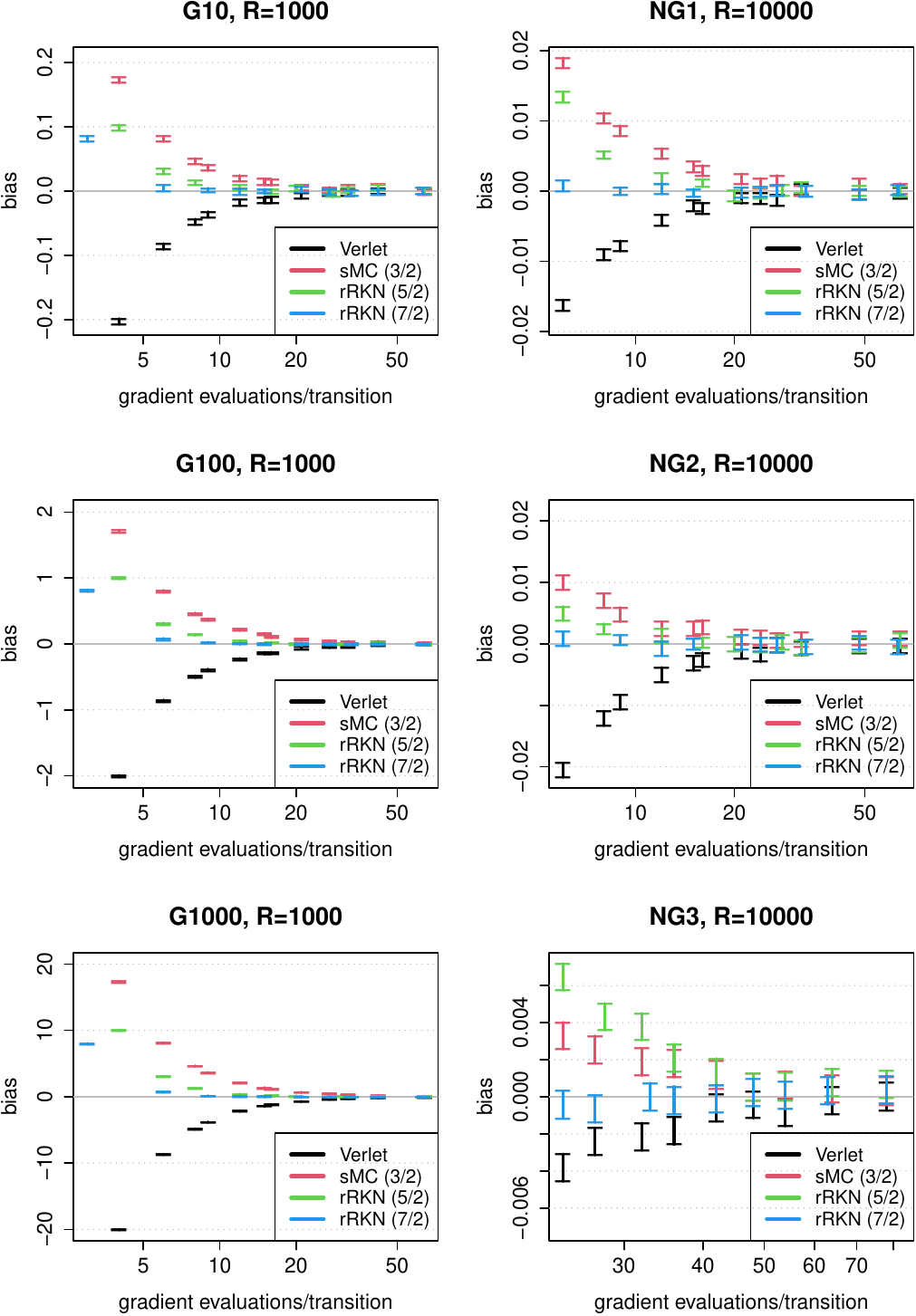}
\caption{95\% confidence intervals for bias of (\ref{numeric_estimator}) associated with uHMC sampling with different time integration schemes and selected number of time steps $T/h$. The different panels correspond to the different target distributions, and the number of simulation replica $R$ specific to each target distribution are indicated. The confidence intervals are plotted according to the number of gradient evaluations $T/h$ required per transition of the $\mathcal{X}_k$ chain on the horizontal axis (log-scale).}
\label{fig:uHMC_KL}
\end{figure}
\begin{figure} 
\centering
\includegraphics[width=0.8\textwidth]{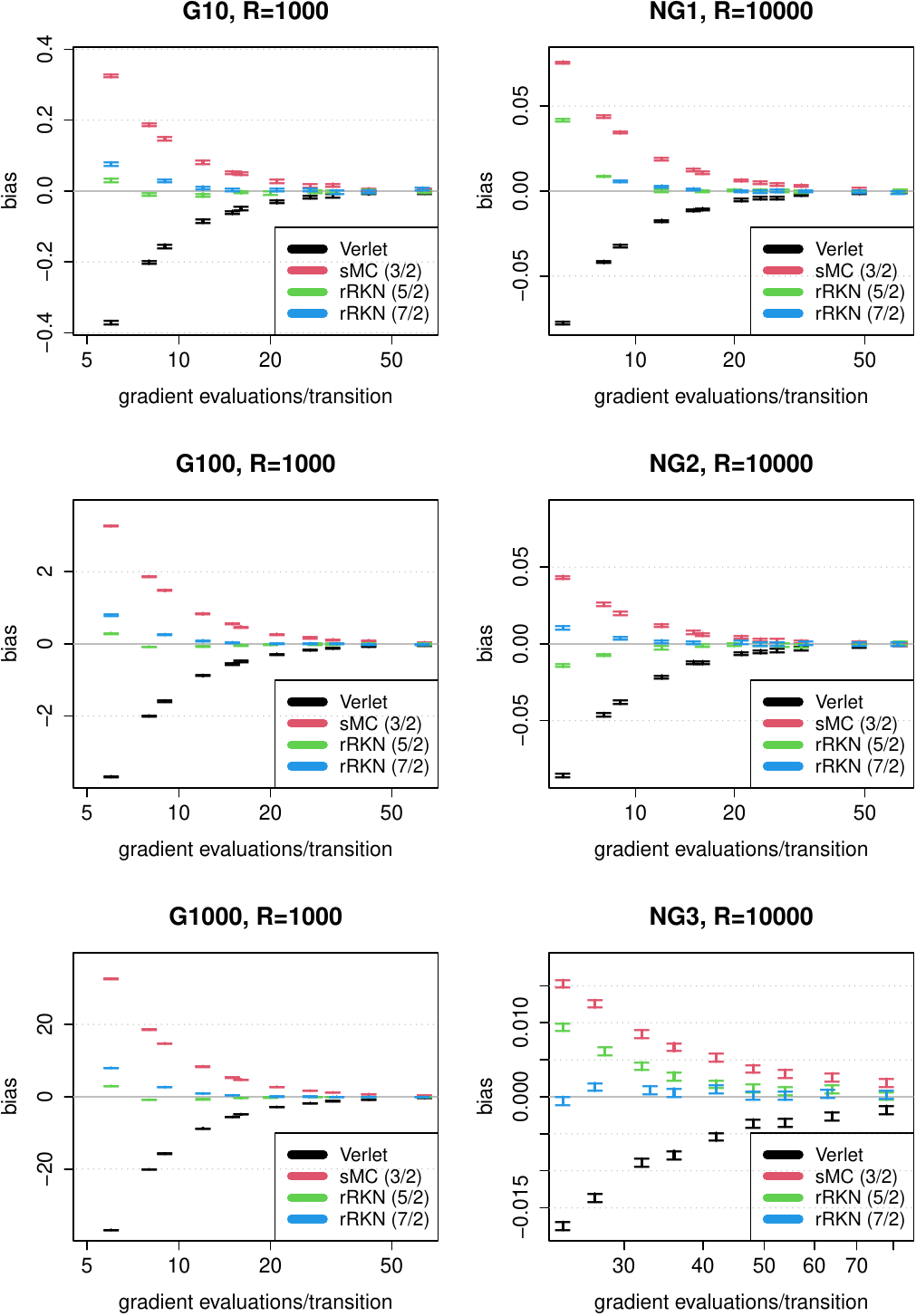}
\caption{
95\% confidence intervals for bias of (\ref{numeric_estimator}) associated with ukLa sampling with different time integration schemes and selected number of time steps $T/h$. The different panels correspond to the different target distributions, and the number of simulation replica $R$ specific to each target distribution are indicated. The confidence intervals are plotted according to the number of gradient evaluations $T/h$ required per transition of the $\mathcal{X}_k$ chain on the horizontal axis (log-scale).}
\label{fig:ukLa_KL}
\end{figure}
\begin{figure} 
\centering
\includegraphics[width=0.99\textwidth]{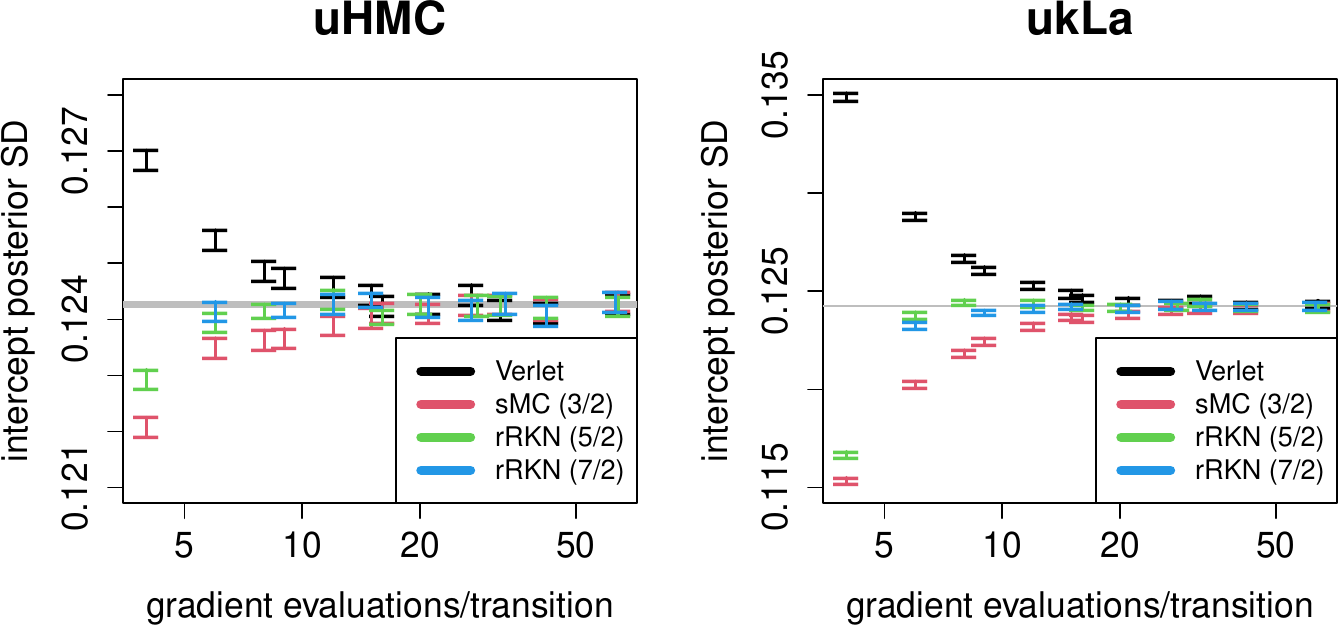}
\caption{Confidence intervals for the $n \rightarrow \infty$ limit of the posterior standard deviation of the intercept term of a logistic regression model applied to the Pima data set. All results are based on $R=1000$ simulation replica with $n=1000$, $T=\pi/2$ and 100 transitions of burn-in. The gray shaded region indicates a 95\% confidence interval calculated from the combined simulations with 64 (63 in the case of 3.5 scheme) gradient evaluations per transition.}
\label{fig:logistic_SD}
\end{figure}

Figure \ref{fig:uHMC_KL} provides 95\% confidence intervals for the bias of (\ref{numeric_estimator}) associated with uHMC for the different target distributions. The intervals are plotted w.r.t.~the number of gradient evaluations used per transition step  of the $\mathcal{X}_k$ chain (which in practical applications is approximately proportional to the overall computational cost of the algorithm). 

It is seen from Figure \ref{fig:uHMC_KL} that except for the challenging NG3 distribution, both proposed schemes produce large improvements in simulation efficiency (measured in bias per number of gradient evaluations) relative to sMC and Verlet. For NG3, the 3.5 scheme again is more efficient than the contending methods, whereas the 2.5 scheme performs rather poorly, at least for low accuracies.
Figure \ref{fig:ukLa_KL} presents similar results as in Figure \ref{fig:uHMC_KL}, but for ukLa. In this case, the proposed schemes appear to uniformly lead to substantial improvements in simulation efficiency.

In addition to the analytically tractable example target distributions above, the different schemes are also applied to sample the (non-Gaussian) posterior distribution of a logistic regression model applied to the Pima data set \cite[see e.g.][]{ripley96}. The data set consists of 532 observations of binary responses and 7 features (in addition to an intercept term). The features were standardized, and sampling (in $q$) was performed subject to an affine preconditioning step. More specifically, $\beta=\hat \beta + \hat \Sigma^{1/2} q$ where $\beta$ is the logistic regression parameter, $\hat \beta$ and $\hat \Sigma$ are the maximum likelihood estimators of $\beta$ and the inverse observed Fisher information of $\beta$ respectively. For all schemes and numbers of time steps, the initial position coordinate $\mathcal{X}_0$ was simulated using ukLA, namely $(\mathcal{X}_0,\tilde{\mathcal{V}})=(Z^u)^{1600}(z_1,z_2)$ with $h=\pi/16$ and rRKN (7/2) integration, where $z_1,z_2 \sim N(0,I_d)$. Note that due to the affine preconditioning step, the target distribution (in $q$-coordinates) is not too far from standard Gaussian, which informs the choice of distribution of $z_1$.

Figure \ref{fig:logistic_SD} provides confidence intervals for the $n\rightarrow \infty$ (but fixed $h$) limit of the posterior standard deviation of the intercept term. Again it is seen that the proposed schemes appear to converge faster toward the approximate asymptote indicated in gray in the Figure. As an example, for 6 gradient evaluations per transition, and with the ukLa sampler, Verlet and sMC are about 3-4 \% off of the approximate asymptote, whereas the corresponding figures are 0.3\% (rRKN 2.5) and 0.7\% (rRKN 3.5) respectively. 

From the numerical illustrations, it may be concluded that the rRKN schemes presented in this paper may provide large savings in computational cost for carrying out un-adjusted MCMC sampling.

\begin{remark}
In Figures~\ref{fig:uHMC_KL} and \ref{fig:ukLa_KL},  note that the asymptotic biases for uHMC and ukLa, based on the Verlet integrator, tend to be negative.  In the Gaussian cases, this can be explained analytically by considering the invariant measure for uHMC and ukLA in the position variable, given by:
$$ \mathcal{N}\left(0, \frac{1}{1 - h^2/4} I_d \right)  $$
Letting $\pi_h$ denote the corresponding density, the estimated bias converges to:
$$\int_{\mathbb{R}^d} (\pi_h(x) - \pi(x)) \log(\pi(x)) = - \frac{d}{2} \frac{h^2}{4 - h^2}.$$
Since $h < 2$ for numerical stability of Verlet, this shows that the asymptotic bias for these samplers is negative in this context. 
\end{remark}

\section{Theory} \label{sec:theory}

While we primarily focus on the 5/2-order $L^2$-accurate scheme in \eqref{eq:rRKN2pt5}, the proofs of the main results on $L^2$-accuracy of rRKN integration schemes rely on three key components:  local bias (Lemma~\ref{lem:rRKN_local_mean_error}),  local mean-squared error (Lemma~\ref{lem:rRKN_local_L2_error}) and numerical stability bounds (Lemma~\ref{lem:apriori} or~\ref{lem:apriori_ukLa}). By modifying these components, the results can be extended  to prove $L^2$-accuracy for other rRKN schemes, demonstrating the flexibility of the theory.  

\subsection{\texorpdfstring{$L^2$}{L2}-accuracy Theorems}

Let  $\opnorm{\cdot}$ denote the operator norm.  In this section, we frequently use the following weighted $\ell_2$-norm on phase space: \[
\wnorm{(x,v)}^2 \ = \ |x|^2 + L^{-1} |v|^2 \;,  \quad (x,v) \in \mathbb{R}^{2d} \;.
\] The corresponding inner product is given by \[ 
\langle (x,v), (\tilde{x},\tilde{v}) \rangle_{\mathsf{w}} = \langle x, \tilde{x} \rangle + L^{-1} \langle v, \tilde{v} \rangle \;,  \quad (x,v), (\tilde{x},\tilde{v})  \in \mathbb{R}^{2d} \;. 
\]   
\begin{assumption} 
\label{B123} The potential $U: \mathbb{R}^d \to \mathbb{R}$ is twice continuously differentiable and satisfies:
\begin{enumerate}[label=\textbf{A.\arabic*}]
\item $U$ has a global minimum at $0$ with $U(0)=0$. \label{B1}
\item $U$ is $L$-gradient Lipschitz continuous, i.e., there exists $L >0$ such that \[
|\nabla U(x) - \nabla U(y)| \ \le \ L \, |x-y| \quad \text{for all $x,y \in \mathbb{R}^d$.} \] \label{B2} 
\vspace{-0.5cm}
\item $U$ is $L_H$-Hessian Lipschitz continuous, i.e., there exists $L_H \ge 0$ such that \[
\opnorm{\nabla^2 U(x) - \nabla^2 U(y)} \ \le \ L_H \, |x-y| \quad \text{for all $x,y \in \mathbb{R}^d$.} \] 
\label{B3} 
\end{enumerate}
\end{assumption}
Under these regularity conditions on $U$, the rRKN integrator in \eqref{eq:rRKN2pt5} achieves $5/2$-order $L^2$-accuracy, while the Verlet integrator   achieves only second-order $L^2$-accuracy.

\begin{theorem}[$L^2$-accuracy of rRKN 2.5 w.r.t. Hamiltonian Flow]
\label{thm:rRKN_L2_accuracy}
Suppose that \ref{B1}-\ref{B3} hold. Let $T>0$ satisfy $L T^2 \le 1/4$  and let $h>0$ satisfy $T/h \in \mathbb{Z}$.  Then for any $x, v \in \mathbb{R}^d$ 
\begin{equation}
\begin{aligned}
&  \max_{t_{k} \le T}  \left( E \left[ \wnorm{(\tilde{Q}_{t_k}(x,v), \tilde{V}_{t_k}(x,v)) - (q_{t_k}(x,v), v_{t_k}(x,v))}^2 \right] \right)^{1/2}  \\
& \qquad \qquad \ \le \ \frac{11}{10} e^{5/2}  \wnorm{(x,v)}  (L^{1/2} h)^{5/2} +  9 e^{5/2} \wnorm{(x,v)}^2 \, L^{1/2} \frac{L_H h^3}{L^{1/4} h^{1/2}} 
   \;.  
   \end{aligned}
   \label{L2_err_rRKN}
\end{equation}
   Here,  $(q_{t_k}(x,v), v_{t_k}(x,v))$  and $(\tilde{Q}_{t_k}(x,v), \tilde{V}_{t_k}(x,v))$ are defined  by \eqref{exact} and \eqref{eq:rRKN2pt5}, respectively.
\end{theorem}

 The proof of Theorem~\ref{thm:rRKN_L2_accuracy} is given in Section~\ref{sec:proofs}, and numerical verification is given in Figure~\ref{fig:strong_accuracy}.

\begin{remark}
For comparison, under only assumptions~\ref{B1}-\ref{B2}, Bou-Rabee \& Marsden \cite{BouRabeeMarsden2022} introduced an rRKN integrator with an $L^2$-accuracy given by: \begin{equation}
   \left( E\left[ |\tilde{Q}_{t_k}(x,v) - q_{t_k}(x,v)|^2 \right] \right)^{1/2} \, \le \,    71  \, ( |x| + L^{-1/2} |v|  ) \, (L^{1/2} h)^{3/2} \;.  \label{L2_err_smc}
\end{equation} 
However, the $L^2$-accuracy of this 1.5 rRKN integrator does not improve if we also assume \ref{B3}.
\end{remark}

\begin{remark} \label{rmk:msefundamental}
 The $5/2$-order $L^2$-accuracy of the rRKN time integrator is reminiscent of the classical Fundamental Theorem for $L^2$-Convergence of Strong Numerical Methods for SDEs, which roughly states: if $p_1 \ge p_2 + 1/2 $ and $p_2 \ge 1/2$ are the local orders of mean and mean-square accuracy (respectively), then the global $L^2$-accuracy of the method is of order $p_2 - 1/2$ \cite[Theorem 1.1.1]{Milstein2021}.  For strong numerical methods for SDEs, this  $(p_2 - 1/2)$-order (as opposed to $(p_2 - 1)$-order) arises due to cancellations in the $L^2$-error expansion, stemming from the independence of the Brownian increments.   Consequently, the expectation of cross terms involving these increments vanishes because they have zero mean. Analogous cancellations (to leading order)  occur in the 2.5 rRKN integrator due to the independence of the sequence of triangular random variables it uses.   
\end{remark}

\begin{figure}[ht!]
\centering
\includegraphics[width=0.49\textwidth]{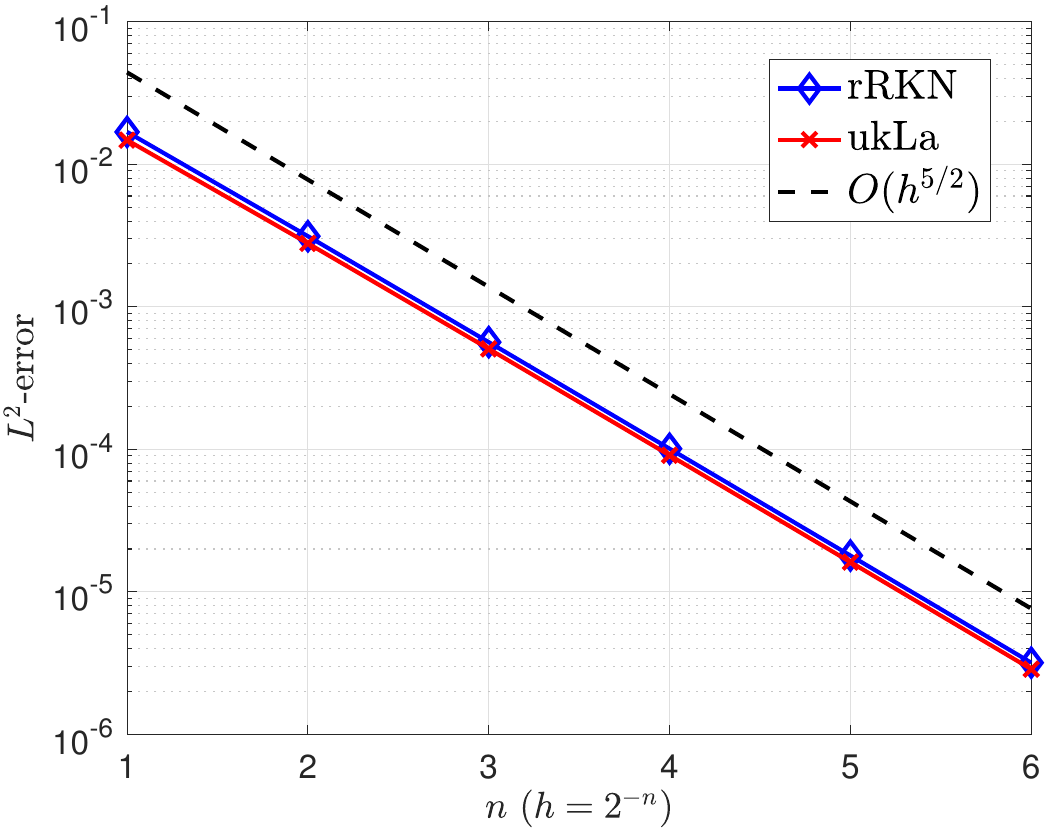}  \hfill
\includegraphics[width=0.49\textwidth]{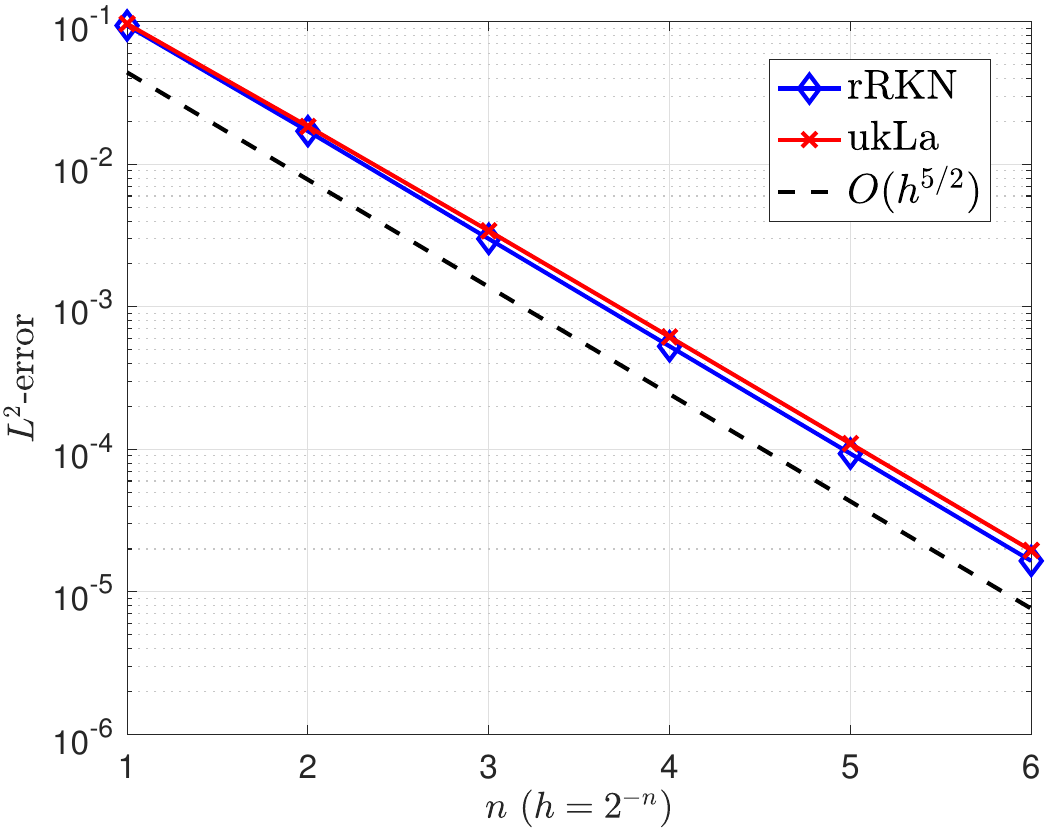}
\caption{ \small {\bf $L^2$-Accuracy  Verification.} Left Image: A plot of the $L^2$-error in $(x,v)$-space of the rRKN time integrator for the linear oscillator with Hamiltonian $H(x,v) = (1/2) (v^2 + x^2)$.  Right Image: A plot of the $L^2$-error in $(x,v)$-space of the rRKN time integrator for a double-well system with Hamiltonian $H(x,v) = (1/2) (v^2 + (1-x^2)^2)$. Both simulations have initial condition $(1,1)$ and unit duration. The time step sizes tested are $2^{-n}$ where $n$ is given on the horizontal axis. The dashed curve is $2^{-5 n/2} = h^{5/2}$ versus $n$. Corresponding results for the ukLa chain (cf.~Theorem~\ref{thm:ukLa_L2_accuracy}) are shown (in red x-marks) with friction factor $\gamma=1/4$.  }
\label{fig:strong_accuracy}
\end{figure}

The next theorem states $5/2$-order $L^2$-accuracy of  ukLa defined in \eqref{OrRKN} relative to the exact splitting defined in \eqref{OX}.

\begin{theorem}[$L^2$-accuracy of rRKN 2.5 based ukLa w.r.t.~Exact Splitting]
\label{thm:ukLa_L2_accuracy}
Suppose that \ref{B1}-\ref{B3} hold. Let $T>0$ satisfy $L T^2 \le 1/4$,  and let $h>0$ satisfy $T/h \in \mathbb{Z}$.  Then for any $z = (x,v) \in \mathbb{R}^{2 d}$, 
\begin{equation}
\begin{aligned}
&  \max_{t_{k} \le T}   \left( E \left[ \wnorm{Z^u_{t_{k}}(z) - Z^e_{t_{k}}(z)  }^2 \right] \right)^{1/2}  \, \le \,  \\
& \qquad  \quad \frac{3}{2} e^{5/2} (L^{1/2} h)^{5/2} \left( \wnorm{z} +  \gamma^{1/2} L^{-1/2} d^{1/2} C_1^{1/2} \sqrt{\frac{e^{2 \gamma T} - 1}{2 \gamma}}   \right)  \\
&  +  14 \sqrt{2} e^{5/2} L^{1/2} \frac{L_H h^3}{L^{1/4} h^{1/2}} \left(   \wnorm{z}^2 + 2 \gamma L^{-1} d C_2^{1/2} T^{1/2}  \sqrt{\frac{e^{4 \gamma T} - 1}{4 \gamma}} \right)  
   \;,  \label{L2_err_ukLa}
   \end{aligned}
   \end{equation}
   where $C_1$ and $C_2$ are numerical constants each coming from an application of the Burkholder-Davis-Gundy inequality.  Here, $Z^u_{t_{k}}(z)$ and $Z^e_{t_{k}}(z)$ appearing in \eqref{L2_err_ukLa} are defined by \eqref{OrRKN} and \eqref{OX}, respectively.  
\end{theorem}

 The proof of Theorem~\ref{thm:ukLa_L2_accuracy} is given in Section~\ref{sec:proofs}, and numerical verification is given in Figure~\ref{fig:strong_accuracy}. 

\begin{remark}
Asymptotic bias bounds for the corresponding unadjusted Hamiltonian and kinetic Langevin Monte Carlo kernels can be directly derived from the $L^2$-accuracy bounds established in Theorem~\ref{thm:rRKN_L2_accuracy} and Theorem~\ref{thm:ukLa_L2_accuracy}.  For example, the results in  \cite[Theorem~6]{BouRabeeMarsden2022} show that, under the assumption of $K$-strong convexity on the potential $U$,  any stable and consistent numerical scheme will inherit the $L^2$-Wasserstein contractivity properties of the exact HMC sampler.  Specifically, the $L^2$-Wasserstein asymptotic bias will involve an accumulation of the finite-time strong error that
can be upper bounded by $1/c$ times the strong accuracy of the method where $c \propto K T^2$ and $T$ is the duration of the Hamiltonian flow.   Recalling that $\mu$ denotes the invariant measure on phase space in \eqref{eq:muk}, it follows from Theorem~\ref{thm:rRKN_L2_accuracy} that the $L^2$-Wasserstein asymptotic bias 
of uHMC with the 2.5 rRKN integration scheme  \eqref{uHMC} is upper bounded by: \begin{equation} 
\label{eq:aymp_bias}
\begin{aligned}
\mathcal{O}\Bigg( \frac{1}{KT^2} \bigg(  & E_{(X,V) \sim \mu} (\wnorm{(X,V)}) (L^{1/2} h)^{5/2} 
 +   E_{(X,V) \sim \mu} (\wnorm{(X,V)}^2)  \, L^{1/2} \frac{L_H h^3}{L^{1/4} h^{1/2}}  \bigg)  \Bigg) \;,
\end{aligned}
\end{equation}
In particular, for the canonical Gaussian measure ($L_H=0$), we have   $E_{(X,V) \sim \mu} (\wnorm{(X,V)}) \propto \sqrt{d}$.  From \eqref{eq:aymp_bias}, it follows that the dimension-dependence of the corresponding complexity is  $\mathcal{O}(d^{1/5})$.  In contrast, for standard HMC, the dimension dependence  is $\mathcal{O}(d^{1/4})$ for this model.  This reduction in dimension dependence quantitatively highlights the benefit of the improved $L^2$-accuracy of the 2.5 rRKN scheme for unadjusted HMC.
\end{remark}

\subsection{Lemmas used in proofs of Theorems}

The proofs of Theorems~\ref{thm:rRKN_L2_accuracy} and \ref{thm:ukLa_L2_accuracy} use a discrete Gr\"{o}nwall inequality, which is stated here for the reader's convenience.  
\begin{lemma}[Discrete Gr\"{o}nwall inequality]
\label{lem:groenwall}
Let $h > 0$ and $\lambda \in \mathbb{R}$ be such that $1+\lambda h > 0$.  Suppose that $(g_k)_{k \in \mathbb{N}_0}$ is a non-decreasing sequence, and $(a_k)_{k \in \mathbb{N}_0}$ satisfies $a_{k+1} \le (1 + \lambda h) a_{k} + g_{k}$ for $k \in \mathbb{N}_0$.  Then it holds \[
a_k \ \le \  (1+\lambda h)^k a_0 + \frac{1}{\lambda h} ( (1+\lambda h)^k - 1) g_{k-1} , \quad \text{for all $k \in \mathbb{N}$} \;.
\]
\end{lemma}

The proof of Theorem~\ref{thm:rRKN_L2_accuracy} uses the following a priori bounds on the rRKN numerical solution.  

\begin{lemma}[A priori bounds for rRKN 2.5] \label{lem:apriori}
Suppose \ref{B1}-\ref{B2} hold. Let $T>0$ satisfy $L T^2 \le 1/4$  and let $h \ge 0$ satisfy $T/h \in \mathbb{Z}$ if $h>0$.  For any $x,v\in\R^d$, we have almost surely
	\begin{align}
	\label{apriori:1}
&\max_{t_{k} \le T}  |\tilde{Q}_{t_k}(x,v)|  \  \leq \   (1+L (T^2+Th)) \max(|x|,|x+T v|) \le \frac{3}{2} (|x| + T |v| ) \;, \\
&\max_{t_{k} \le T}  |\tilde{V}_{t_k}(x,v)|   \  \le \ 
|v| + L T (1+L (T^2+T h))  \max(|x|,|x+T v|)  \ \le \ \frac{11}{8} |v| + \frac{3}{4}  L^{1/2} |x| \;. \label{apriori:2}
\end{align} 
Here, $(\tilde{Q}_{t_k}(x,v), \tilde{V}_{t_k}(x,v))$ is defined by \eqref{eq:rRKN2pt5}.
\end{lemma}
The proof of Lemma~\ref{lem:apriori} is nearly identical to the proof of Lemma~3.1 of \cite{BoEbZi2020} and hence is omitted.  The following special case of  Lemma~\ref{lem:apriori}   is frequently used  in the proofs that follow: \begin{align}
\label{apriori:2a}
\sup_{0 \le s \le h} |v_s(x,v)|^2 \ &\le \ 4 |v|^2 + L |x|^2 \;.
\end{align}
The proof of Theorem~\ref{thm:ukLa_L2_accuracy} uses the following upper bounds on the moments of ukLa. 

\begin{lemma}[Moment bounds for ukLa with rRKN 2.5] \label{lem:apriori_ukLa}
Suppose \ref{B1}-\ref{B2} hold. Let $T>0$ satisfy $L T^2 \le 1/4$  and let $h \ge 0$ satisfy $T/h \in \mathbb{Z}$ if $h>0$.  For any $z = (x,v) \in\R^d$ and $\ell \in \mathbb{N}$, we have 	\begin{align}	\label{eq:apriori_ukLa}
&\max_{t_{k} \le T}  E\left( \wnorm{Z^u_{t_k}(z)}^{2 \ell} \right)   \  \leq \  \frac{18^{\ell}}{2} \wnorm{z}^{2 \ell} + \frac{9^{\ell}}{2}  (2 \gamma L^{-1})^{\ell} d^{\ell} C_{\ell}  T^{\ell-1} \frac{e^{2 \ell \gamma T} - 1}{2 \ell \gamma}   \;,
\end{align}
where $C_{\ell}$ is a numerical constant that only depends on $\ell$. Here, $Z^u_{t_k}(z)$ is defined by \eqref{OrRKN}.
\end{lemma}

The remaining  lemmas needed to prove Theorem~\ref{thm:rRKN_L2_accuracy} and Theorem~\ref{thm:ukLa_L2_accuracy} follow.

\begin{lemma}[Lipschitz continuity of exact flow] Suppose \ref{B1}-\ref{B2} hold.  For all $(x,v), (\tilde{x},\tilde{v})  \in \mathbb{R}^{2d}$, and for all $T>0$ such that $L T^2 \le 1/4$, \begin{equation}
\sup_{0 \le s \le T} \wnorm{\varphi_s(x,v) - \varphi_s(\tilde{x},\tilde{v})}^2 \ \le \ (1+6 L^{1/2} T) \wnorm{(x,v) - (\tilde{x},\tilde{v})}^2 \;. 
\end{equation}
Here, $\varphi_t(x,v)$ denotes the exact flow of \eqref{exact}.  
\label{lem:lipschitz_exact_flow}
\end{lemma}

\begin{lemma}[Local bias of rRKN integator]  Suppose \ref{B1}-\ref{B3} hold.  Let $h>0$ satisfy $L h^2 \le 1/4$. Then for all $(x,v) \in \mathbb{R}^{2d}$, the following bound holds \begin{equation}
\begin{aligned}
& \wnorm{E (\tilde{Q}_h(x,v),\tilde{V}_h(x,v)) - (q_h(x,v),v_h(x,v)) }^2 \\ 
& \qquad  \le \ \frac{2}{9} \left[  ( L h^2 )^4 |x|^2 + ( L h^2 )^3 h^2  |v|^2  + 3 ( L_H h^3 )^2 (L h^2) L |x|^4 + 4 (L_H h^3)^2 h^2 |v|^4 \right]   \;. 
\end{aligned}
\end{equation}
Here, $(q_h(x,v),v_h(x,v))$ and $(\tilde{Q}_h(x,v),\tilde{V}_h(x,v))$ are defined by \eqref{exact} and \eqref{eq:rRKN2pt5}, respectively.
\label{lem:rRKN_local_mean_error}
\end{lemma}

\begin{lemma}[Local $L^2$-error of rRKN integator]  Suppose \ref{B1}-\ref{B3} hold.   Let $h>0$ satisfy $L h^2 \le 1/4$.  Then for all $(x,v) \in \mathbb{R}^{2d}$, the following bound holds almost surely \begin{equation}
\begin{aligned}
&  \wnorm{(\tilde{Q}_h(x,v),\tilde{V}_h(x,v)) - (q_h(x,v),v_h(x,v)) }^2  \\
 & \qquad  \ \le \ \frac{7}{4} \left[ (L^{1/2} h)^6 \left(  |x|^2 +  L^{-1} |v|^2 \right) +L  (L_H h^3)^2   \left(|x|^4 + 10  L^{-2} |v|^4\right) \right]     \;. 
 \end{aligned}
\end{equation}
Here, $(q_h(x,v),v_h(x,v))$ and $(\tilde{Q}_h(x,v),\tilde{V}_h(x,v))$ are defined by \eqref{exact} and \eqref{eq:rRKN2pt5}, respectively.
\label{lem:rRKN_local_L2_error}
\end{lemma}

\subsection{Proofs of \texorpdfstring{$L^2$}{L2}-accuracy Theorems}

\label{sec:proofs}

\begin{proof}[Proof of Theorem~\ref{thm:rRKN_L2_accuracy}] 
For any $k \in \mathbb{N}_0$, let $\mathcal{F}_{t_k}$ denote the sigma-algebra of events up to  $t_k$ generated by the i.i.d.~sequence  $(\mathcal{U}_i )_{i \in \mathbb{N}_0}$ up to step $k$.  For the proof, it is helpful to recall that $h > 0$ is a time step size, $(q_{t_{k}}, v_{t_{k}})$ and $(\tilde{Q}_{t_{k}}, \tilde{V}_{t_{k}} )$ are defined by \eqref{exact} and \eqref{eq:rRKN2pt5} respectively, and that $\varphi_h$ denotes the exact flow of \eqref{exact}.   

To estimate the global $L^2$-error, we begin by rewriting the squared $L^2$-error over one time integration step by introducing a zero $ - \varphi_h(\tilde{Q}_{t_k}, \tilde{V}_{t_k}) + \varphi_h(\tilde{Q}_{t_k}, \tilde{V}_{t_k})$, applying the tower property of the  conditional expectation, and expanding as follows: \begin{align*}
 \rho_{t_{k+1}}^2 \ &:= \  E \left( \wnorm{(\tilde{Q}_{t_{k+1}}, \tilde{V}_{t_{k+1}} ) - (q_{t_{k+1}}, v_{t_{k+1}})  }^2 \right)   \nonumber \\
&  \ = \  E \left( \wnorm{(\tilde{Q}_{t_{k+1}}, \tilde{V}_{t_{k+1}} ) - \varphi_h(\tilde{Q}_{t_k}, \tilde{V}_{t_k}) + \varphi_h(\tilde{Q}_{t_k}, \tilde{V}_{t_k}) - (q_{t_{k+1}}, v_{t_{k+1}})  }^2 \right)  \nonumber \\
& \ = \ 
E \left( \wnorm{(\tilde{Q}_{t_{k+1}}, \tilde{V}_{t_{k+1}} ) - \varphi_h(\tilde{Q}_{t_k}, \tilde{V}_{t_k}) }^2 \right)  + E \left( \wnorm{\varphi_h(\tilde{Q}_{t_k}, \tilde{V}_{t_k}) - (q_{t_{k+1}}, v_{t_{k+1}})  }^2 \right)  \nonumber \\
& \qquad + 2 E \left[ \left\langle E \left( (\tilde{Q}_{t_{k+1}}, \tilde{V}_{t_{k+1}} ) \mid \mathcal{F}_{t_k} \right) - \varphi_h(\tilde{Q}_{t_k}, \tilde{V}_{t_k})    , \varphi_h(\tilde{Q}_{t_k}, \tilde{V}_{t_k}) - (q_{t_{k+1}}, v_{t_{k+1}})  \right\rangle_w \right] \nonumber \\
& \ \le \   
E \left( \wnorm{(\tilde{Q}_{t_{k+1}}, \tilde{V}_{t_{k+1}} ) - \varphi_h(\tilde{Q}_{t_k}, \tilde{V}_{t_k}) }^2 \right)  + E \left( \wnorm{\varphi_h(\tilde{Q}_{t_k}, \tilde{V}_{t_k}) - (q_{t_{k+1}}, v_{t_{k+1}})  }^2 \right)  \nonumber \\
& \qquad + 2 E \left[  \wnorm{E \left( (\tilde{Q}_{t_{k+1}}, \tilde{V}_{t_{k+1}} ) \mid \mathcal{F}_{t_k} \right) - \varphi_h(\tilde{Q}_{t_k}, \tilde{V}_{t_k}) }   \wnorm{ \varphi_h(\tilde{Q}_{t_k}, \tilde{V}_{t_k}) - (q_{t_{k+1}}, v_{t_{k+1}})}   \right] \nonumber 
\;.
\end{align*}
where in the last step we applied the Cauchy-Schwarz inequality to the cross term. Using Young's inequality for products with $\varepsilon = L^{1/2} h$, we obtain the following bound:  \begin{align}  \rho_{t_{k+1}}^2   \ \le \  \rn{1} + \rn{2} + \rn{3}  \quad  
\label{L2_error_decomposition} 
\end{align}
where we have introduced
\begin{align}
\rn{1}  \ &:= \  (1 + L^{1/2} h) \ E \wnorm{\varphi_h(\tilde{Q}_{t_k}, \tilde{V}_{t_k})- \varphi_h(q_{t_k}, v_{t_k})}^2 \;, \nonumber \\
\rn{2}  \ &:= \ \frac{1}{L^{1/2} h} \, E \wnorm{E\left( (\tilde{Q}_{t_{k+1}}, \tilde{V}_{t_{k+1}}) \mid \mathcal{F}_{t_k} \right) - \varphi_h(\tilde{Q}_{t_k}, \tilde{V}_{t_k}) }^2 \;, \nonumber \\
\rn{3}  \ &:= \
E \wnorm{ (\tilde{Q}_{t_{k+1}}, \tilde{V}_{t_{k+1}}) - \varphi_h(\tilde{Q}_{t_k}, \tilde{V}_{t_k}) }^2 \;. \nonumber 
\end{align}
  By Lemma~\ref{lem:lipschitz_exact_flow} on the Lipschitz continuity of the exact flow,
\begin{align}
\rn{1} \ \le \ (1+ L^{1/2} h) (1+6 L^{1/2} h) \,  \rho_{t_k}^2  \overset{(L^{1/2} h \le 1/2)}{\le}  (1+ 10 L^{1/2} h) \, \rho_{t_k}^2 \;. \label{L2:1} 
\end{align}
By Lemma~\ref{lem:rRKN_local_mean_error} on the local bias of the rRKN integrator, and using $L^{1/2} h \le 1/2$,
\begin{align}
\rn{2} \ \le \
\frac{1}{9} E \left[  ( L^{1/2} h )^6 (|\tilde{Q}_{t_k}|^2 + L^{-1} |\tilde{V}_{t_k}|^2 )  +  ( L_H h^3 )^2 L ( 3 |\tilde{Q}_{t_k}|^4 + 4  L^{-2} |\tilde{V}_{t_k}|^4  ) \right]
\;. \label{L2:2} 
\end{align}
By Lemma~\ref{lem:rRKN_local_L2_error} on the local $L^2$-error of the rRKN integrator, 
\begin{equation}
\begin{aligned}
&\rn{3} \ \le \ \frac{7}{4} E \left[ (L^{1/2} h)^6  (|\tilde{Q}_{t_k}|^2 + L^{-1} |\tilde{V}_{t_k}|^2) \right]  + 10 L  (L_H h^3)^2   E\left[ (  |\tilde{Q}_{t_k}|^4 +   L^{-2} |\tilde{V}_{t_k}|^4 )  \right]  \;. 
\end{aligned}
\label{L2:3}
\end{equation}
Combining \eqref{L2:1}, \eqref{L2:2}, and \eqref{L2:3}, we get: \begin{align}
  &   \rho_{t_{k+1}}^2 \ \le \ (1+ 10 \, L^{1/2} \, h) \, \rho_{t_k}^2 + 2 \, (L^{1/2} h)^6  \, \mathfrak{M}_{2}  + 18 \, L \, (L_H h^3)^2 \, \mathfrak{M}_{4}  \;, \nonumber  
\end{align}
where we have introduced \begin{align*}
\mathfrak{M}_{2} \ &:= \ E\left[ \left(\max_{t_k \le T}|\tilde{Q}_{t_k}| \right)^2 + L^{-1} \left(\max_{t_k \le T} |\tilde{V}_{t_k}| \right)^2 \right] \;,  \\ \mathfrak{M}_{4} \ &:= \  E\left[ \left(\max_{t_k \le T} |\tilde{Q}_{t_k}|\right)^4 + L^{-2} \left(\max_{t_k \le T}|\tilde{V}_{t_k}| \right)^4 \right] \;.
\end{align*}
We now apply discrete Gr\"{o}nwall's inequality (Lemma~\ref{lem:groenwall}) with  \[ \lambda = 10 L^{1/2} \;, \quad a_k = \rho_{t_k}^2 \;, \quad \text{and}  \quad g_k = 2  (L h^2)^3  \mathfrak{M}_{2}  + 18 L (L_H h^3)^2 \mathfrak{M}_{4} 
\]  yielding: 
\begin{align}
\label{eq:postgroenwall}
\rho_{t_k}^2 \ &\le \ \frac{1}{10} \, \exp(10 L^{1/2} T) \, \left[ 2 \, (L^{1/2} h)^5 \, \mathfrak{M}_{2} + 18 \, L^{1/2} L_H^2 h^5 \, \mathfrak{M}_{4} \right] \;.
\end{align}
Next, we bound $\mathfrak{M}_{2}$ using  Lemma~\ref{lem:apriori}, Jensen's inequality with respect to the convex function $\mathsf{x} \mapsto \mathsf{x}^2$, and the condition $L T^2 \le 1/4$,    \begin{align}
\mathfrak{M}_{2} 
& \ \le \ \left( \frac{3}{2} (|x|+T |v|) \right)^2 + L^{-1} \left( \frac{11}{8} |v| + \frac{3}{4} L^{1/2} |x| \right)^2 \nonumber \\
& \ \le \ 2 \left(  \frac{9}{4} + \frac{9}{16} \right) |x|^2 +2 \left(  \frac{9}{4} LT^2 +  \left(\frac{11}{8} \right)^2  \right) L^{-1} |v|^2  \nonumber \\
& \ \le \ \frac{45}{8} |x|^2 + \frac{157}{32}  L^{-1} |v|^2  \ \le \  \frac{45}{8} \, \wnorm{(x,v)}^2 \;. \label{eq:2ndmombnd} 
\end{align}
 Similarly, using Lemma~\ref{lem:apriori}, Jensen's inequality with respect to the convex function $\mathsf{x} \mapsto \mathsf{x}^4$, and the condition $L T^2 \le 1/4$, we bound  $\mathfrak{M}_{4}$ as
\begin{align}
\mathfrak{M}_{4}
& \ \le \ \left( \frac{3}{2} (|x|+T |v|) \right)^4 + L^{-2} \left( \frac{11}{8} |v| + \frac{3}{4} L^{1/2} |x| \right)^4 \nonumber \\
& \ \le \ 2^3 \left[ \left( \frac{3}{2} \right)^4 (|x|^4 + T^4 |v|^4)  + \left( \frac{3}{4} \right)^4 |x|^4 + \left( \frac{11}{8} \right)^4 L^{-2} |v|^4 \right]\nonumber \\
& \ \le \ 2^3 \left[ \left(  \left( \frac{3}{2} \right)^4  + \left( \frac{3}{4} \right)^4 \right) |x|^4 + \left( \left( \frac{3}{2} \right)^4 L^2 T^4 + \left( \frac{11}{8} \right)^4 \right) L^{-2} |v|^4  \right] \nonumber \\
& \ \le \ \frac{1377}{32} \left( |x|^4 +  L^{-2} |v|^4 \right)   \;.
\label{eq:4thmombnd}
\end{align}
Finally, we insert \eqref{eq:2ndmombnd} and \eqref{eq:4thmombnd} into \eqref{eq:postgroenwall}, take square roots, and  apply the subadditivity of the square root to obtain \eqref{L2_err_rRKN}, as required.  
\end{proof}

Given analogous local bias (Lemma~\ref{lem:rRKN_local_mean_error}), local mean-squared error (Lemma~\ref{lem:rRKN_local_L2_error}) and numerical stability bounds (Lemma~\ref{lem:apriori}), this proof can be straightforwardly adapted to establish global $L^2$-accuracy for other rRKN schemes.  Specifically, the estimates for the terms appearing in \eqref{L2_error_decomposition} follow directly from these lemmas, as shown in \eqref{L2:2}, \eqref{L2:3}, and the numerical stability bounds in \eqref{eq:2ndmombnd} and \eqref{eq:4thmombnd}.

\begin{proof}[Proof of Theorem~\ref{thm:ukLa_L2_accuracy}]
For any $k \in \mathbb{N}_0$, let $\mathcal{F}_{t_k}$ denote the sigma-algebra of events up to  $t_k$ generated by the independent i.i.d.~sequences  $(\xi_i )_{i \in \mathbb{N}_0}$ and $(\mathcal{U}_i )_{i \in \mathbb{N}_0}$ up to step $k$. 
For any $z = (x,v) \in \mathbb{R}^{2 d}$ and $k \in \mathbb{N}_0$, let $Z^e_{t_k}(z) = (X_{t_{k}}^e(z), V_{t_{k}}^e(z))$ and $Z^u_{t_k}(z) = (X_{t_{k}}^u(z), V_{t_{k}}^u(z)) $ be defined by \eqref{OX} and \eqref{OrRKN}, respectively.  Since the OU-substep is a contraction in  $\wnorm{}$, \begin{align}  \rho_{t_{k+1}}^2 \ &:= \ E \left( \wnorm{Z_{t_{k+1}}^e(z)  - Z_{t_{k+1}}^u(z)  }^2 \right) \nonumber \\
\ &= \ E \left( \wnorm{O_h(\xi_k) \circ \varphi_h(Z_{t_{k}}^e(z) ) - O_h(\xi_k) \circ \Theta_h(\mathcal{U}_k) ( Z_{t_{k}}^u(z) )  }^2 \right)  \nonumber \\
\ &\le \ E \left( \wnorm{\varphi_h(Z_{t_{k}}^e(z) ) -  \Theta_h(\mathcal{U}_k) ( Z_{t_{k}}^u(z) )  }^2 \right)  
\ \le \  \rn{1} + \rn{2} + \rn{3}  \quad \text{where} 
\nonumber \\
\rn{1}  \ &:= (1 + L^{1/2} h) \ E \wnorm{\varphi_h(Z^e_{t_k}(z))- \varphi_h(Z_{t_k}^u(z))}^2 \;, \nonumber \\
\rn{2}  \ &:= \ \frac{1}{L^{1/2} h} \, E \wnorm{E\left( \Theta_h(\mathcal{U}_k)(Z_{t_k}^u(z)) \mid \mathcal{F}_{t_k} \right) - \varphi_h(Z_{t_k}^u(z) ) }^2 \;, \nonumber \\
\rn{3}  \ &:= \
E \wnorm{ \Theta_h(\mathcal{U}_k)(Z_{t_k}^u(z)) - \varphi_h(Z_{t_k}^u(z)) }^2 \;. \nonumber 
\end{align}
  By Lemma~\ref{lem:lipschitz_exact_flow} on the Lipschitz continuity of the exact flow,
\begin{align}
\rn{1} \ \le \ (1+ L^{1/2} h) (1+6 L^{1/2} h) \, \rho_{t_k}^2  \overset{(L^{1/2} h \le 1/2)}{\le}  (1+ 10 L^{1/2} h) \, \rho_{t_k}^2 \;. \label{L2::1} 
\end{align}
By Lemma~\ref{lem:rRKN_local_mean_error} on the local mean error of the rRKN integrator, and using the condition $L^{1/2} h \le 1/2$, we have
\begin{align}
\rn{2} \ &\le \   
\frac{1}{9} E \left(  ( L h^2 )^3 \wnorm{Z^u_{t_k}}^2 + 4 ( L_H h^3 )^2 L ( |X^u_{t_k}|^4 + L^{-2} |V^u_{t_k}|^4  ) \right) 
\;. \label{L2::2} 
\end{align}
By Lemma~\ref{lem:rRKN_local_L2_error} on the local $L^2$-error of the rRKN integrator, 
\begin{align}
\rn{3} \ &\le \ \frac{7}{4} E \left( (L h^2)^3 \wnorm{Z^u_{t_k}}^2 + 10 (L_H h^3)^2  L ( |X^u_{t_k}|^4 +  L^{-2} |V^u_{t_k}|^4 )  \right)  \;. \label{L2::3}
\end{align}
Combining \eqref{L2::1} - \eqref{L2::3} yields \begin{align}
  &   \rho_{t_{k+1}}^2 \ \le \ (1+ 10 L^{1/2} h) \, \rho_{t_k}^2  + 2  (L h^2)^3  \, E \left(\wnorm{Z_{t_k}^u}^2 \right)  +  18 L (L_H h^3)^2   \, E \left(\wnorm{Z_{t_k}^u}^4 \right) \;. \nonumber  
\end{align}
Applying, in turn, Lemma~\ref{lem:apriori_ukLa},  Lemma~\ref{lem:groenwall}, and $L h^2 \le L T^2 \le 1/4$, and then simplifying gives the required result.
\end{proof}

The proofs of Theorem~\ref{thm:rRKN_L2_accuracy} and Theorem~\ref{thm:ukLa_L2_accuracy} share the same components:  local bias (Lemma~\ref{lem:rRKN_local_mean_error}), local mean-squared error (Lemma~\ref{lem:rRKN_local_L2_error}) and numerical stability bounds (Lemma~\ref{lem:apriori} or Lemma~\ref{lem:apriori_ukLa}).

\subsection{Proofs of Lemmas}

\begin{proof}[Proof of Lemma~\ref{lem:apriori_ukLa}]
This proof involves a careful comparison of the ukLa chain to the corresponding free chain defined for any $z = (x,v) \in \mathbb{R}^{2d}$ and $i \in \mathbb{N}_0$ by the recurrence relation
\[
Z^0_{t_{i+1}}(z) \ = \ O_h(\xi_i) \circ \theta_h^{(A)}( Z^0_{t_i}(z) ) \;, \quad Z^0_{t_0}(z) \ =  \  z  \ \in \ \mathbb{R}^{2 d} \;.  
\]  where $\theta_h^{(A)} (z) := (x + h v, v)$.  The rationale for this comparison is that the moments of the free chain can be estimated more explicitly. To see this, we introduce the following weighted $\ell_1$-norm: \[
\wonenorm{(x,v)} \ = \ |x| + L^{-1/2} |v| \;, \quad (x,v) \in \mathbb{R}^{2d} \;,
\] 
and note that $\{Z^0_{t_i}(z) \}_{i \in \mathbb{N}_0} = \{(X^0_{t_i}(z), V^0_{t_i}(z)) \}_{i \in \mathbb{N}_0}  $ can be interpolated by \[
d X^0_t \ = \ V^0_{\lfloor t \rfloor} dt \;, \qquad d V^0_t \ = \ - \gamma V^0_t + \sqrt{2 \gamma} d B_t \;, 
\]
where $\lfloor t \rfloor \, := \,  \sup \{ s\in h\Z \ : \ s\leq t \}$ and $(X^0_0, V^0_0)  \, = \,  (x,v) \, \in \,  \mathbb{R}^{2d}$.   Hence, \begin{align*}
X^0_t  \ &=  \  x + \left(\int_0^t e^{-\gamma \lfloor s \rfloor} ds \right) v  + \sqrt{2 \gamma} \int_0^t \int_0^{\lfloor s \rfloor} e^{-\gamma ( \lfloor s \rfloor - r ) } d B_r ds \;, \\
V_t^0 \ &= \ e^{-\gamma t} v + \sqrt{2 \gamma} \int_0^t e^{-\gamma (t-s) } d B_s \;,
\end{align*}
and therefore, almost surely for $t \le T$, it holds that \[ |V^0_t| \le e^{-\gamma t} |v| + \sqrt{2 \gamma} e^{-\gamma t} \sup_{t \le T} \left| \int_0^t e^{\gamma s } d B_s \right| \;,  \] and hence, \begin{align} 
\wonenorm{Z^0_{t_k}(z)} \ &\le \  (1+ L^{1/2} T) \wonenorm{z} + \sqrt{2 \gamma} (T+ L^{-1/2})  \sup_{t \le T} \left|\int_0^t e^{\gamma s} d B_s  \right| \nonumber \\
\ &\le \  \frac{3}{2} \left( \wonenorm{z} + \sqrt{2 \gamma}  L^{-1/2} \sup_{t \le T} \left|\int_0^t e^{\gamma s} d B_s  \right| \right)  \;.  \label{eq:ukLa_free_moment}
\end{align}
where we used $L^{1/2} T \le 1/2$.  

\smallskip

The comparison argument upper bounds $\Delta_{t_i}  :=  \wonenorm{Z^u_{t_i}(z)-Z^0_{t_i}(z)}$ using the discrete Gr\"{o}nwall inequality in Lemma~\ref{lem:groenwall}.  By the triangle inequality, \begin{align}
\Delta_{t_{i+1}}  \ &\le \ \rn{1} + \rn{2} \qquad \text{where} \label{eq:delta_ukLa} \\
 \rn{1} \ &= \ \wonenorm{Z^u_{t_{i+1}}(z) -O_h(\xi_i) \circ \theta_h^{(A)}(Z_{t_i}^u(z)) } \;,  \quad \text{and} \nonumber \\
 \rn{2} \ &= \ 
\wonenorm{O_h(\xi_i) \circ \theta_h^{(A)}(Z_{t_i}^u(z)) - O_h(\xi_i) \circ \theta_h^{(A)}(Z_{t_i}^0(z)) } \;. \nonumber 
\end{align}
To upper bound $\rn{2}$, we use the definitions of $\theta_h^{(A)}$ and the OU-map \eqref{eq:OU} 
\begin{align}
\rn{2} \ &\le \ | X_{t_i}^u(z) - X_{t_i}^0(z)| + (L^{1/2} h) L^{-1/2}     | V_{t_i}^u(z) - V_{t_i}^0(z)| 
\nonumber  \\
& \quad +   e^{-\gamma h} L^{-1/2} | V_{t_i}^u(z) - V_{t_i}^0(z)| \nonumber \\
\ &\le \ (1+ L^{1/2} h) \Delta_{t_i} \;. \label{eq:rn2_ukLa}
\end{align}
To upper bound $\rn{1}$, we additionally use the definition of the ukLa transition step \eqref{ukLa},  $L^{1/2} h \le 1/2$, and  $|\mathbf{F}(x)| \le L |x|$ (which follows from Assumptions~\ref{B1}-\ref{B2}), \begin{align}
& \rn{1} \ \le \   \left| \frac{h^2}{2}  \Tilde{\mathbf{F}}_{t_i}^- + \frac{h^2}{6 \mathcal{U}_i} \left(  \Tilde{\mathbf{F}}_{t_i}^+ - \Tilde{\mathbf{F}}_{t_i}^- \right)  \right| + L^{-1/2} \left| h \tilde{\mathbf{F} }_{t_i}^-  + \frac{h}{2 \mathcal{U}_i} \left(  \Tilde{\mathbf{F}}_{t_i}^+ - \Tilde{\mathbf{F}}_{t_i}^- \right)  \right| \;, \nonumber \\   
& ~~ \overset{\ref{B2}}{\le}   \frac{L h^2}{2} |X^u_{t_i} | + \frac{L  h^2}{6} \left|h V_{t_i}^u + \frac{\mathcal{U}_i h^2}{2} \tilde{\mathbf{F} }_{t_i}^-  \right| + (L^{1/2} h) |X^u_{t_i}| + \frac{L^{1/2} h}{2} \left|h V_{t_i}^u + \frac{\mathcal{U}_i h^2}{2} \tilde{\mathbf{F} }_{t_i}^-  \right| \;,  \nonumber  \\
& ~~ \ \le \  \left(  \frac{(L^{1/2} h)^2}{2} + \frac{(L^{1/2} h)^4}{12} + L^{1/2} h + \frac{(L^{1/2} h)^3}{4} \right)  |X^u_{t_i}| \nonumber \\
& \qquad + \left( \frac{(L^{1/2} h)^3}{6} + \frac{(L^{1/2} h)^2}{2}  \right) L^{-1/2} |V^u_{t_i}| \;, \nonumber \\
& ~~ \ \le \ \frac{4}{3} L^{1/2} h \wonenorm{Z^u_{t_i}(z)} \ \le \ \frac{4}{3} L^{1/2} h \Delta_{t_i} +  \frac{4}{3} L^{1/2} h \wonenorm{Z^0_{t_i}(z)} \;. 
\label{eq:rn1_ukLa}
\end{align}
Inserting \eqref{eq:rn1_ukLa} and \eqref{eq:rn2_ukLa} back into \eqref{eq:delta_ukLa} yields \begin{align}
& \Delta_{t_{i+1}} \ \le \ \left( 1 +  \frac{7}{3} L^{1/2} h \right)  \Delta_{t_i} +  \frac{4}{3} L^{1/2} h \sup_{k:~t_{k} \le T}  \wonenorm{Z^0_{t_k}(z)}  \  \le \  \frac{4}{7} \, e^{7/6} \sup_{k:~t_{k} \le T}  \wonenorm{Z^0_{t_k}(z)} \nonumber 
\end{align} 
where in the last step we used the discrete Gr\"{o}nwall inequality from Lemma~\ref{lem:groenwall}. Therefore, almost surely \begin{align}
& \wonenorm{Z_{t_i}^u(z)} \ \le \ \Delta_{t_i} +\wonenorm{ Z^0_{t_i}(z)} \  \le \  \left(1 + \frac{4}{7} \, e^{7/6} \right)  \sup_{k:~t_{k} \le T}  \wonenorm{Z^0_{t_k}(z)} \nonumber \\
& \qquad \ \le \  \frac{9}{2}  \left( \wonenorm{z} + \sqrt{2 \gamma}  L^{-1/2} \sup_{t \le T} \left|\int_0^t e^{\gamma s} d B_s  \right| \right) \;. \label{ukLa_wonebound} 
\end{align}
Let $(B^1_t)_{t \ge 0}$ be the first component of the $d$-dimensional Brownian motion $(B_t)_{t \ge 0}$ and define the martingale $M_t := \int_0^t e^{\gamma s } d B_s^1$. By the Burkholder-Davis-Gundy inequality  w.r.t.~the martingale $M_t$, i.e., \begin{align}
\label{ieq:BDG}
E \left( \sup_{t \le T} |M_t|^{2 \ell} \right) \, &\le \,   C_{\ell} \left(\int_0^T e^{ \gamma s} ds \right)^{\ell} \, \le \, C_{\ell}  T^{\ell-1} \int_0^T e^{2 \gamma \ell s} ds  \, \le \, C_{\ell}  T^{\ell-1} \frac{e^{2 \ell \gamma T} - 1}{2 \ell \gamma}  
\end{align}
where $C_{\ell}$ is a constant that depends only on $\ell$.  Then \begin{align}
& E\left( \wnorm{Z^u_{t_k}(z)}^{2 \ell} \right)   \  \leq \ E\left( \wonenorm{Z^u_{t_k}(z)}^{2 \ell} \right)   \nonumber    \\
& \qquad \overset{\eqref{ukLa_wonebound}}{\leq} \   \left(\frac{9}{2} \right)^{2 \ell}  2^{2 \ell - 1} E\left( \wonenorm{z}^{2 \ell} +   \left( \sqrt{2 \gamma/L}  \sup_{t \le T} \left| \int_0^t e^{\gamma s} d B_s \right| \right)^{2 \ell} \right)  \nonumber \\
& \qquad \ \le \ \frac{9^{\ell}}{2} \wonenorm{z}^{2 \ell} + \frac{9^{\ell}}{2}  (2 \gamma L^{-1})^{\ell} d^{\ell} E \left( \sup_{t \le T} |M_t|^{2 \ell} \right)  \nonumber \\
& \qquad \overset{\eqref{ieq:BDG}}{\le}  \frac{18^{\ell}}{2} \wnorm{z}^{2 \ell} + \frac{9^{\ell}}{2}  (2 \gamma L^{-1})^{\ell} d^{\ell} C_{\ell}  T^{\ell-1} \frac{e^{2 \ell \gamma T} - 1}{2 \ell \gamma} \;.
\label{eq:ukLa_moment_bound}
\end{align}
as required.
\end{proof}

\begin{proof}[Proof of Lemma~\ref{lem:lipschitz_exact_flow}]
Fix $t>0$ such that $L t^2 \le 1/4$.  As shorthand, let \[
(x_s^{(1)},v_s^{(1)}):=(x_s(x,v),v_s(x,v)) ~~ \text{and} ~~ (x_s^{(2)},v_s^{(2)}):=(x_s(\tilde{x}, \tilde{v}),v_s(\tilde{x}, \tilde{v})) \;, ~~  s \ge 0 \;. 
\]
Using \eqref{exact}, we have
\begin{align}
  & \wnorm{\varphi_t(x,v) -  \varphi_t(\tilde{x},\tilde{v})}^2   \ = \     |x_t^{(1)} - x_t^{(2)}|^2 + L^{-1}  |v_t^{(1)} - v_t^{(2)}|^2   \nonumber \\
  & \qquad \ = \    \left| x - \tilde{x} + \int_0^t (v_s^{(1)} - v_s^{(2)}) ds \right|^2 + L^{-1} \, \left| v-\tilde{v} + \int_0^t ( \mathbf{F}(x_s^{(1)}) - \mathbf{F}(x_s^{(2)}) ) ds \right|^2  \nonumber \\
 & \qquad \ = \ \wnorm{(x,v) - (\tilde{x},\tilde{v})}^2 + \rn{1} + \rn{2} \nonumber
 \end{align}
 where we have introduced
 \begin{align}
 \rn{1} \ &:= \  2 \left\langle x - \tilde{x}, \int_0^t (v_s^{(1)} - v_s^{(2)}) ds \right\rangle + 2 L^{-1} \left\langle v - \tilde{v}, \int_0^t (\mathbf{F}(x_s^{(1)})-\mathbf{F}(x_s^{(2)})) ds \right\rangle  \;, \nonumber \\
 \rn{2} \ &:= \  \left| \int_0^t (v_s^{(1)} - v_s^{(2)}) ds \right|^2 + L^{-1} \left| \int_0^t (\mathbf{F}(x_s^{(1)})-\mathbf{F}(x_s^{(2)})) ds \right|^2 \;. \nonumber 
\end{align}
 By applying the Cauchy-Schwarz and Young's inequalities, we bound the term $\rn{1}$ as follows: \begin{align}
\rn{1}  &\overset{\ref{B2}}{\le}  2 L^{1/2} \left[  |x-\tilde{x}| \left( L^{-1/2}  \int_0^t |v_s^{(1)} - v_s^{(2)}| ds \right) +    \left(  \int_0^t |x_s^{(1)} - x_s^{(2)}| ds \right)  L^{-1/2} |v-\tilde{v}|  \right] \nonumber \\
\ &\le \ 2 L^{1/2} t \left( |x-\tilde{x}|^2 + L^{-1} |v-\tilde{v}|^2 \right) \nonumber \\
& \qquad + \frac{L^{1/2} t}{2} \left( \left( t^{-1} \int_0^t |x_s^{(1)} - x_s^{(2)}| ds \right)^2 + L^{-1} \left( t^{-1} \int_0^t |v_s^{(1)} - v_s^{(2)}| ds \right)^2 \right)   \nonumber \\
\ &\le \ 2 L^{1/2} t \wnorm{(x,v) - (\tilde{x},\tilde{v})}^2 + \frac{L^{1/2} t}{2} \left( t^{-1} \int_0^t \wnorm{\varphi_s(x,v) - \varphi_s(\tilde{x},\tilde{v})}^2  ds  \right)  \nonumber \\
\ &\le \ 2 L^{1/2} t \wnorm{(x,v) - (\tilde{x},\tilde{v})}^2 + \frac{L^{1/2} t}{2} \left( \sup_{0 \le s \le t} \wnorm{\varphi_s(x,v) - \varphi_s(\tilde{x},\tilde{v})}^2   \right)  \;. \label{LC:1}
\end{align}
For $\rn{2}$,  we similarly apply the Cauchy-Schwarz inequality and use the condition $L^{1/2} t \le 1/2$, 
\begin{align}
& \rn{2}  \overset{\ref{B2}}{\le} L t \left( \int_0^t \wnorm{\varphi_s(x,v) - \varphi_s(\tilde{x},\tilde{v})}^2 ds   \right) \overset{(L^{1/2} t \le 1/2)}{\le}   \frac{L^{1/2} t}{2} \, \sup_{0 \le s \le t} \wnorm{\varphi_s(x,v) - \varphi_s(\tilde{x},\tilde{v})}^2    \;. \label{LC:2}
\end{align}
Combining \eqref{LC:1} and \eqref{LC:2}, we get the following: \begin{align}
 & \sup_{0 \le s \le t} \wnorm{\varphi_s(x,v) - \varphi_s(\tilde{x},\tilde{v})}^2 \le \frac{1 + 2 L^{1/2} t}{1-L^{1/2} t} \wnorm{(x,v) - (\tilde{x},\tilde{v})}^2  \ \le \   ( 1 + 6 L^{1/2} t) \wnorm{(x,v) - (\tilde{x},\tilde{v})}^2  \nonumber 
 \end{align}
 where in the last step we used the elementary inequality $(1+2 \mathsf{x})/(1-\mathsf{x}) \le 1 + 6 \mathsf{x} $ valid for $\mathsf{x} \in [0,1/2]$.  This gives the required result.
\end{proof}

\begin{proof}[Proof of Lemma~\ref{lem:rRKN_local_mean_error}] 
Let $\mathbf{F} \equiv -\nabla U$, $\mathbf{H} \equiv -\nabla^2 U$,
and $(x_s,v_s):=(x_s(x,v),v_s(x,v))$.  By \eqref{exact} and \eqref{eq:rRKN2pt5}, note that \begin{align}
& \wnorm{E(\tilde{Q}_h,\tilde{V}_h) - (q_h,v_h)}^2 \ = \ \rn{1} + \rn{2} 
 \end{align}
where we have introduced
 \begin{align}
\rn{1} \  &:= \  \left| \frac{h}{3} \int_0^h \big( \mathbf{F}(x+s v + \frac{s^2}{2} \mathbf{F}(x)) - \mathbf{F}(x) \big) ds - \int_0^h \frac{(h-s)^2}{2} \mathbf{H}(x_s) \cdot v_s ds \right|^2  \;, \nonumber \\
\rn{2} \ &:= \ L^{-1}  \left| \int_0^h \big( \mathbf{F}(x+s v + \frac{s^2}{2} \mathbf{F}(x)) - \mathbf{F}(x) \big) ds - \int_0^h (h-s) \mathbf{H}(x_s) \cdot v_s ds \right|^2  \;. \nonumber 
 \end{align}
Integrating by parts and applying the Cauchy-Schwarz inequality yields \begin{align}
 \rn{1} \ &= \     \left| \frac{h}{3} \int\limits_0^h (h-s)  \mathbf{H}(x+s v + \frac{s^2}{2} \mathbf{F}(x)) \cdot (v + s \mathbf{F}(x) ) ds - \int\limits_0^h \frac{(h-s)^2}{2} \mathbf{H}(x_s) \cdot v_s ds \right|^2 \nonumber \\
 \ &\le \ \rn{1}_a + \rn{1}_b + \rn{1}_c + \rn{1}_d   \quad \text{where} \nonumber  \\ 
 \rn{1}_a \ &:= \  4 \left| \frac{h}{3} \int_0^h (h-s) s \mathbf{H}(x+s v + \frac{s^2}{2} \mathbf{F}(x)) \cdot \mathbf{F}(x)   ds  \right|^2 \nonumber \;, \\
 \rn{1}_b \ &:= \  4 \left| \frac{h}{3} \int_0^h (h-s) \big( \mathbf{H}(x+s v + \frac{s^2}{2} \mathbf{F}(x)) - \mathbf{H}(x) \big) \cdot v    ds  \right|^2 \nonumber \;, \\
 \rn{1}_c \ &:= \  4 \left|  \int_0^h \frac{(h-s)^2}{2} \big( \mathbf{H}(x_s) - \mathbf{H}(x) \big) \cdot v   ds  \right|^2 \nonumber \;, \\
 \rn{1}_d \ &:= \  4  \left| \int_0^h \frac{(h-s)^2}{2} \mathbf{H}(x_s) \cdot \big( v_s - v \big) ds \right|^2 \nonumber \;.\end{align}
By Assumptions~\ref{B1}-\ref{B2}, $\opnorm{\mathbf{H}(x)} \le L$ and $|\mathbf{F}(x)| \le L |x|$.  Therefore, by Cauchy-Schwarz inequality, \begin{align}
 \rn{1}_a \ &\le \ 4 \, \frac{h^2}{9} \int_0^h (h-s)^2 s^2 ds \, \int_0^h | \mathbf{H}(x + s v + \frac{s^2}{2} \mathbf{F}(x)) \cdot \mathbf{F}(x)|^2 ds \nonumber \\
 \ &\le \  \frac{4}{9} \frac{h^8}{30} L^4 |x|^2 \le \frac{1}{60} (L h^2)^4 |x|^2 \label{ME:1a}
\end{align}
Additionally invoking Assumption~\ref{B3} gives
\begin{align}
 \rn{1}_b \ &\le \ 4 \frac{h^2}{9} \int_0^h (h-s)^2 ds \int_0^h \left| \big( \mathbf{H}(x+s v + \frac{s^2}{2} \mathbf{F}(x)) - \mathbf{H}(x) \big) \cdot v  \right|^2 ds \nonumber \\
 &\le 4 \frac{h^2}{9} \frac{h^3}{3} L_H^2 |v|^2 \int_0^h \left| s v + \frac{s^2}{2}  \mathbf{F}(x) \right|^2 ds   \nonumber \\
&\le 4 \frac{h^5}{27}  L_H^2 |v|^2 \int_0^h (2 s^2 |v|^2 + s^4  |\mathbf{F}(x)|^2 ) ds  \nonumber \\ 
&\le 
\frac{4}{27}  (L_H h^3)^2   \left( \frac{h^2}{3} |v|^4 +  \frac{1}{10}  2 \, h \, |v|^2 \cdot (L h^2)^{3/2} \, L^{1/2} \, |x|^2 \right)   \nonumber \\ 
&\le 
\frac{1}{6}  (L_H h^3)^2   \left( h^2 |v|^4 +  \frac{1}{10} (L h^2)^3 L |x|^4 \right)   
\label{ME:1b}
\end{align}

\begin{align}
 \rn{1}_c \ &\le \  \int_0^h (h-s)^4 ds \int_0^h |\big( \mathbf{H}(x_s) - \mathbf{H}(x) \big) \cdot v  |^2 ds \le \frac{h^5}{5} L_H^2 |v|^2 \int_0^h \left|\int_0^s v_r dr \right|^2 ds \nonumber \\
&\le \frac{h^5}{5} L_H^2 |v|^2 \int_0^h s \int_0^s |v_r|^2 dr ds   \le \frac{h^5}{5} L_H^2 |v|^2 \sup_{0 \le r \le h} |v_r|^2 \int_0^h s^2 ds  \nonumber \\
&\overset{\eqref{apriori:2a}}{\le} \frac{1}{15} (L_H h^3)^2  (4 h^2 |v|^4 + L h^2 |v|^2 |x|^2) \nonumber \\
&\le \frac{1}{2} (L_H h^3)^2  (h^2 |v|^4 + (L h^2)  L |x|^4)  \label{ME:1c}
\end{align}

\begin{align}
 \rn{1}_d \ &\le \   \int_0^h (h-s)^4 ds \int_0^h | \mathbf{H}(x_s) \cdot \big( v_s - v \big)  |^2 ds  \le  L^2 \frac{h^5}{5}  \int_0^h |  v_s - v   |^2 ds \nonumber \\
\ &\le \ L^2 \frac{h^5}{5}  \int_0^h \left| \int_0^s \mathbf{F}(x_r) dr \right|^2 ds  \le L^4 \frac{h^5}{5} \int_0^h s \int_0^s \left|x+\int_0^r v_{u} du \right|^2 dr ds \nonumber \\
\ &\le \ L^4 \frac{h^5}{5} \int_0^h s \int_0^s \left( 2|x|^2 + 2 \left| \int_0^r v_u du \right|^2 \right) dr ds  \nonumber \\ 
\ &\le \  L^4 \frac{h^5}{5} \int_0^h s \int_0^s \left( 2|x|^2 + 2 r \int_0^r |v_u|^2 du \right) dr ds \nonumber \\
\ &\le \ L^4 \frac{h^5}{5} \int_0^h \left( 2 s^2 |x|^2 + \frac{2}{3} s^4 \sup_{0 \le r \le h} |v_r|^2 \right) ds  \nonumber \\
\ & \overset{\eqref{apriori:2a}}{\le} \ \frac{1}{5}    (L h^2)^4 \left( \frac{2}{3}  |x|^2 + \frac{2}{15} h^2 (4 |v|^2 + L |x|^2) \right) \nonumber \\ 
\ &\le \ \frac{8}{75}  (L h^2)^4 h^2 |v|^2 + (L h^2)^4  \left( \frac{2}{15}  + \frac{2}{75} (L h^2) \right)  |x|^2    \label{ME:1d}
\end{align}

Similarly, we obtain \begin{align}
 \rn{2} \ &= \   L^{-1} \left|  \int\limits_0^h (h-s)  \mathbf{H}(x+s v + \frac{s^2}{2} \mathbf{F}(x)) \cdot (v + s \mathbf{F}(x) ) ds - \int\limits_0^h (h-s) \mathbf{H}(x_s) \cdot v_s ds \right|^2 \nonumber \\
 \ &\le \ \rn{2}_a + \rn{2}_b  \quad \text{where} \nonumber  \\ 
 \rn{2}_a \ &:= \ 2 \, L^{-1} \left|  \int_0^h (h-s) ( \mathbf{H}(x+s v + \frac{s^2}{2} \mathbf{F}(x)) - \mathbf{H}(x_s) ) \cdot (v + s \mathbf{F}(x) ) ds  \right|^2 \nonumber \;, \\
 \rn{2}_b \ &:= \ 2 \, L^{-1} \left|  \int_0^h (h-s) \mathbf{H}(x_s) \cdot \big( v_s - (v + s \mathbf{F}(x) ) \big) ds \right|^2 \nonumber \;.
\end{align}

By Assumptions~\ref{B1}-\ref{B3},
\begin{align}
 \rn{2}_a \ &\le \  2 \, L^{-1} \int_0^h (h-s)^2 ds \int_0^h |   ( \mathbf{H}(x+s v + \frac{s^2}{2} \mathbf{F}(x)) - \mathbf{H}(x_s) ) \cdot (v + s \mathbf{F}(x) ) |^2 ds \nonumber \\
&\le 2 L^{-1} \frac{h^3}{3} L_H^2 \int_0^h \left| \int_0^s \frac{(s-r)^2}{2} \mathbf{H}(x_r) \cdot v_r dr \right|^2 |v + s \mathbf{F}(x)|^2 ds \nonumber \\
&\le 2 L^{-1} \frac{h^3}{3} L_H^2 \int_0^h \left( \int_0^s \frac{(s-r)^4}{4} dr \right) \left(  \int_0^s |\mathbf{H}(x_r) \cdot v_r|^2 dr \right)  |v + s \mathbf{F}(x)|^2 ds \nonumber \\
&\le \frac{2}{3} L L_H^2 h^3 \sup_{0 \le r \le h} |v_r|^2 \int_0^h \frac{s^6}{20} (2 |v|^2 + 2 s^2 |\mathbf{F}(x)|^2 ) ds  \nonumber \\ 
&\le \frac{2}{3} L L_H^2 h^3 \sup_{0 \le r \le h} |v_r|^2 \int_0^h \frac{s^6}{20} (2 |v|^2 + 2 s^2 L^2 |x|^2 ) ds  \nonumber \\
&\le  \frac{2}{3} L L_H^2 h^3 \sup_{0 \le r \le h} |v_r|^2 \left( \frac{h^7}{70} |v|^2 +   \frac{h^9}{90} L^2 |x|^2 \right) \nonumber \\
&\overset{\eqref{apriori:2a}}{\le} \frac{2}{3} L L_H^2 h^3 \left( 4 |v|^2 + L |x|^2 \right) \left( \frac{h^7}{70} |v|^2 +   \frac{h^9}{90} L^2 |x|^2 \right) \nonumber \\
&\le \frac{1}{17} \left( (L_H h^3)^2 (L h^2) h^2 |v|^4 + (L_H h^3)^2 (L h^2) L|x|^4 \right) \;, \label{ME:2a}
\end{align}
where in the last step we used that $L h^2 \le 1/4$. \begin{align}
 \rn{2}_b \ &\le \  2 \, L^{-1} \int_0^h (h-s)^2 ds \int_0^h |  \ \mathbf{H}(x_s) \cdot \big( v_s - (v + s \mathbf{F}(x) ) \big) |^2 ds  \nonumber \\ 
 &\le \ \frac{2}{3} L h^3 \int_0^h | \int_0^s (s-r) \mathbf{H}(x_r) \cdot v_r dr |^2 ds  \nonumber \\
 \ &\le \   \frac{2}{3} L h^3 \int_0^h  \left( \int_0^s (s-r)^2 dr \right)  \left( \int_0^s |\mathbf{H}(x_r) \cdot v_r  |^2 dr \right) ds \nonumber \\
 &\le \ \frac{2}{3} L^3  h^3 \sup_{0 \le r \le h} |v_r|^2 \int_0^h \frac{s^4}{3} ds  \nonumber \\
  &\overset{\eqref{apriori:2a}}{\le} \frac{2}{45} (L h^2)^3   ( 4 h^2 |v|^2 + L h^2 |x|^2 ) \;. 
 \label{ME:2b}
\end{align}

Combining \eqref{ME:1a} - \eqref{ME:1d}, \eqref{ME:2a}, and  \eqref{ME:2b} yields, \begin{equation*}
\begin{aligned}
& \wnorm{E(\tilde{Q}_h,\tilde{V}_h) - (q_h,v_h)}^2  \\
& \qquad \ \le \  \frac{2}{9} \left(  ( L h^2 )^4 |x|^2 + ( L h^2 )^3 h^2  |v|^2 + 3 ( L_H h^3 )^2 (L h^2) L |x|^4 + 4 (L_H h^3)^2 h^2 |v|^4 \right) 
\end{aligned}
\end{equation*}
as required.
\end{proof}

\begin{proof}[Proof of Lemma~\ref{lem:rRKN_local_L2_error}]  Let $\mathbf{F} \equiv -\nabla U$, $\mathbf{H} \equiv -\nabla^2 U$,
 $\mathcal{U} \sim \text{Triangular}(0,1)$, and set $(x_s,v_s):=(x_s(x,v),v_s(x,v))$.  As detailed below, a stronger almost-sure result can be proven by using $\mathcal{U} \le 1$.
 
\smallskip

By \eqref{exact} and \eqref{eq:rRKN2pt5}, \begin{align}
& \wnorm{(\tilde{Q}_h,\tilde{V}_h) - (q_h,v_h)}^2 \ = \ \rn{1} + \rn{2} \quad \text{where} \nonumber \\
\rn{1} &:= \left| \frac{h^2}{6 \mathcal{U}} \big( \mathbf{F}(x+\mathcal{U} h v + \frac{(\mathcal{U} h)^2}{2} \mathbf{F}(x)) - \mathbf{F}(x) \big)  - \int_0^h \frac{(h-s)^2}{2} \mathbf{H}(x_s) \cdot v_s ds \right|^2 \;,  \nonumber \\
\rn{2} &:= L^{-1}  \left| \frac{h}{2 \mathcal{U}} \big( \mathbf{F}(x+\mathcal{U} h v + \frac{(\mathcal{U} h)^2}{2} \mathbf{F}(x)) - \mathbf{F}(x) \big) - \int_0^h (h-s) \mathbf{H}(x_s) \cdot v_s ds \right|^2  \;.  \nonumber 
\end{align}
By the fundamental theorem of calculus for line integrals,  \begin{align}
 \rn{1} \ &= \  \left| \frac{h^2}{6} \int\limits_0^{h} \mathbf{H}(x+ \mathcal{U} s v + \frac{(\mathcal{U} s)^2}{2} \mathbf{F}(x)) \cdot (v + \mathcal{U} s \mathbf{F}(x)) ds  - \int\limits_0^h \frac{(h-s)^2}{2} \mathbf{H}(x_s) \cdot v_s ds \right|^2 \nonumber \\
 \ &\le \ \rn{1}_a + \rn{1}_b \quad \text{where} \nonumber  \\ 
 \rn{1}_a \ &:= \  2  \left| \frac{h^2}{6} \int_0^{h} \mathbf{H}(x+ \mathcal{U} s v + \frac{(\mathcal{U} s)^2}{2} \mathbf{F}(x)) \cdot (v + \mathcal{U} s \mathbf{F}(x)) ds \right|^2   \nonumber \;, \\
 \rn{1}_b \ &:= \  2  \left|  \int_0^h \frac{(h-s)^2}{2} \mathbf{H}(x_s) \cdot v_s ds \right|^2 \nonumber \;.
 \end{align}
Similarly, we obtain  \begin{align}
 \rn{2} \ &= \  L^{-1}  \left| \frac{h}{2} \int\limits_0^h \big( \mathbf{H}(x+\mathcal{U} s v + \frac{(\mathcal{U} s)^2}{2} \mathbf{F}(x)) \cdot (v + \mathcal{U} s \mathbf{F}(x) )  \big) ds - \int\limits_0^h (h-s) \mathbf{H}(x_s) \cdot v_s ds \right|^2  \nonumber \\
 \ &\le \ \rn{2}_a + \rn{2}_b + \rn{2}_c + \rn{2}_d \quad \text{where} \nonumber  \\ 
 \rn{2}_a \ &:= \ 4 L^{-1} \left| \frac{h}{2} \int_0^h s \,  \mathcal{U}  \,  \mathbf{H}(x+\mathcal{U} s v + \frac{(\mathcal{U} s)^2}{2} \mathbf{F}(x)) \cdot \mathbf{F}(x)  ds  \right|^2 \nonumber \;, \\
 \rn{2}_b \ &:= \   4 L^{-1}  \left| \frac{h}{2} \int_0^h \big( \mathbf{H}(x+\mathcal{U} s v + \frac{(\mathcal{U} s)^2}{2} \mathbf{F}(x)) - \mathbf{H}(x) \big) \cdot v    ds  \right|^2 \nonumber \;, \\
 \rn{2}_c \ &:= \ 4 L^{-1}  \left|  \int_0^h (h-s) \big( \mathbf{H}(x_s) - \mathbf{H}(x) \big) \cdot v_s ds \right|^2 \nonumber \;, \\
 \rn{2}_d \ &:= \ 4 L^{-1}  \left|  \int_0^h (h-s)  \mathbf{H}(x)  \cdot (v_s - v)  ds \right|^2   \nonumber \;.\end{align}
By Assumption~\ref{B1}-\ref{B2}, $\opnorm{\mathbf{H}} \le L$ and $|\mathbf{F}(x)| \le L |x|$.  Therefore, by Cauchy-Schwarz inequality, we obtain the bounds \begin{align}
 \rn{1}_a \ &\le \ \frac{h^5}{18} \int_0^h \left| \mathbf{H}(x+ \mathcal{U} s v + \frac{(\mathcal{U} s)^2}{2} \mathbf{F}(x)) \cdot (v + \mathcal{U} s \mathbf{F}(x)) \right|^2 ds  \nonumber \\
&\le \frac{L^2 h^5}{18} \int_0^h | v + \mathcal{U} s \mathbf{F}(x) |^2 ds  \le \frac{L^2 h^5}{18} \int_0^h ( 2 | v |^2 + 2 s^2 | \mathbf{F}(x) |^2 ) ds   \nonumber \\ 
&\le \frac{L^2 h^6}{18}  \left( 2 | v |^2 + \frac{2}{3}  L^2 h^2 |x|^2 \right)   \le \frac{(L h^2)^3}{18}  \left( 2 L^{-1} | v |^2 + \frac{1}{6}   |x|^2 \right)   \label{MSE:1a}
\end{align}

\begin{align}
 \rn{1}_b \ &\le \ 2 \int_0^h \frac{(h-s)^4}{4} ds \int_0^h | \mathbf{H}(x_s) \cdot v_s |^2 ds \ \le \  \frac{L^2 h^6}{10} \sup_{0 \le s \le h} |v_s|^2 \nonumber \\
 \ &\overset{\eqref{apriori:2a}}{\le} \ 
 \frac{(L h^2)^3}{10}   (4 L^{-1} |v|^2 +  |x|^2) 
 \label{MSE:1b}
\end{align}

\begin{align}
 \rn{2}_a \ &\le \  L^{-1} h^2 \int_0^h s^2 ds  \int_0^h | \mathbf{H}(x+\mathcal{U} s v + \frac{(\mathcal{U} s)^2}{2} \mathbf{F}(x)) \cdot \mathbf{F}(x) |^2 ds \nonumber \\
 \ &\le \ \frac{(L h^2)^3}{3} |x|^2 \label{MSE:2a}
\end{align}

\begin{align}
 \rn{2}_b \ &\le \ L^{-1} h^3 \int_0^h \left|\left( \mathbf{H}(x+\mathcal{U} s v + \frac{(\mathcal{U} s)^2}{2} \mathbf{F}(x)) - \mathbf{H}(x) \right) \cdot v \right|^2 ds  \nonumber \\
 &\le L^{-1} h^3 L_H^2 |v|^2 \int_0^h \left| \mathcal{U} s v + \frac{(\mathcal{U} s)^2}{2} \mathbf{F}(x) \right|^2 ds   \nonumber \\
&\le L^{-1} h^3 L_H^2 |v|^2 \int_0^h \left( 2 s^2 |v|^2 + \frac{s^4}{2} |\mathbf{F}(x) |^2 \right) ds \nonumber \\
&\le L^{-1} h^3 L_H^2 |v|^2  \left( \frac{2}{3} h^3 |v|^2 + \frac{h^5}{10} L^2 |x|^2 \right)  \nonumber \\
&\le L^{-1}  L_H^2 h^6  \left( \frac{2}{3}  |v|^4 + 2 \frac{h^2}{20} L^2 |v|^2 |x|^2 \right) \nonumber \\
&\le   (L_H h^3)^2   \left( L^{-1} |v|^4 +  \frac{3}{400} (L h^2)^2 L |x|^4 \right)   \label{MSE:2b}
\end{align}

\begin{align}
 \rn{2}_c \ &\le \ 4 L^{-1} \int_0^h (h-s)^2 ds \int_0^h | ( \mathbf{H}(x_s) - \mathbf{H}(x) ) \cdot v_s |^2 ds \nonumber \\
\ & \le \   4 L^{-1} \frac{h^3}{3} L_H^2 \int_0^h |x_s - x|^2 |v_s|^2 ds  \nonumber \\
 \ & \le \   4 L^{-1} \frac{h^3}{3} L_H^2 \int_0^h \left|\int_0^s v_r dr \right|^2 |v_s|^2 ds \le 4 L^{-1} \frac{h^3}{3} L_H^2 \left(\sup_{0 \le r \le h} |v_r|^2 \right)^2 \int_0^h s^2 ds \nonumber \\
\  &   \overset{\eqref{apriori:2a}}{\le}  \frac{4}{9}  L^{-1} (L_H h^3)^2 ( 4 |v|^2 + L |x|^2)^2  
\ \le \    \frac{4}{9}   (L_H h^3)^2 ( 32  L^{-1} |v|^4 + 2 L |x|^4)  \label{MSE:2c}
\end{align}

\begin{align}
 \rn{2}_d \ &\le \  4 L^{-1} \int_0^h (h-s)^2 ds \int_0^h | \mathbf{H}(x)  \cdot (v_s - v)  |^2 ds  \le 4 L  \frac{h^3}{3}   \int_0^h |v_s - v |^2 ds \nonumber \\
\ &\le \ 4 L \frac{h^3}{3}  \int_0^h \left| \int_0^s \mathbf{F}(x_r) dr \right|^2 ds  \le 4 L^3 \frac{h^3}{3} \int_0^h s \int_0^s \left|x+\int_0^r v_{u} du \right|^2 dr ds \nonumber \\
\ &\le \ 4 L^3 \frac{h^3}{3} \int_0^h s \int_0^s \left( 2|x|^2 + 2 \left| \int_0^r v_u du \right|^2 \right) dr ds   \nonumber  \\ 
\ &\le  \ 4 L^3 \frac{h^3}{3} \int_0^h s \int_0^s \left( 2|x|^2 + 2 r \int_0^r |v_u|^2 du \right) dr ds \nonumber \\
\ &\le \ 4 L^3 \frac{h^3}{3} \int_0^h \left( 2 s^2 |x|^2 + \frac{2}{3} s^4 \sup_{0 \le r \le h} |v_r|^2 \right) ds  \nonumber  \\
\ &\le \ \frac{4}{3}    (L h^2)^3  \left( \frac{2}{3}  |x|^2 + \frac{2}{15} h^2 \sup_{0 \le r \le h} |v_r|^2 \right) \nonumber \\
\ & \overset{\eqref{apriori:2a}}{\le} \ \frac{4}{3}    (L h^2)^3  \left( \frac{2}{3}  |x|^2 + \frac{2}{15} h^2 (4 |v|^2 + L |x|^2) \right) \nonumber  \\
\ &\le \ \frac{32}{45}  (L h^2)^3 h^2 |v|^2 + (L h^2)^3  \left( \frac{8}{9} + \frac{8}{45} L h^2 \right) |x|^2 \;.  \label{MSE:2d}
\end{align}

Combining \eqref{MSE:1a}-\eqref{MSE:1b} and \eqref{MSE:2a}-\eqref{MSE:2d} and using the condition $L h^2 \le 1/4$, we get almost surely \begin{align}
& \wnorm{(\tilde{Q}_h,\tilde{V}_h) - (q_h,v_h)}^2 \nonumber  \\
& \quad \ \le \ \frac{7}{5} \left( (L h^2)^3  |x|^2 + (L h^2)^3 ( h^2 + L^{-1} ) |v|^2 + (L_H h^3)^2  L |x|^4 + 12 (L_H h^3)^2  L^{-1} |v|^4 \right)  \nonumber \\
& \quad \ \le \ \frac{7}{4} \left[ (L^{1/2} h)^6 \left(  |x|^2 +  L^{-1} |v|^2 \right) + (L_H h^3)^2  L \left(|x|^4 + 10  L^{-2} |v|^4\right) \right]  \nonumber
\end{align}
as required.  
\end{proof}

\printbibliography

\end{document}